\documentclass[a4paper,12pt,twosided]{amsart}

\usepackage{amsmath}
\usepackage{amsfonts}
\usepackage{dsfont}
\usepackage{amsthm}
\usepackage{tikz}
\usepackage{enumerate}
\usepackage[T1]{fontenc}
\usepackage[utf8]{inputenc}
\usepackage{accents}

\newcommand\blfootnote[1]{%
  \begingroup
  \renewcommand\thefootnote{}\footnote{#1}%
  \addtocounter{footnote}{-1}%
  \endgroup
}

\newtheorem{definition}{Definition}[section]
\newtheorem{lemma}[definition]{Lemma}
\newtheorem{theorem}[definition]{Theorem}
\newtheorem{corollary}[definition]{Corollary}
\newtheorem{example}[definition]{Example}
\newtheorem{remark}[definition]{Remark}
\newtheorem{conjecture}[definition]{Conjecture}
\DeclareMathOperator\supp{supp}
\newcommand{\sgn}{\operatorname{sgn}}

\author{Florian Stebegg}
\title[Model-Independent Pricing of Asian Options]{Model-Independent Pricing of Asian Options
via Optimal Martingale Transport}

\begin{document}

\begin{abstract}
In this article we discuss the problem of calculating optimal model-independent (robust) bounds for the price of
Asian options with discrete and continuous averaging. We will give geometric
characterisations of the maximising and the minimising pricing model for certain
types of Asian options in discrete and continuous time. In discrete time the problem is reduced to
finding the optimal martingale transport for the cost function $|x+y|$.
In the continuous time case we consider the cases with one and two given marginals. 
We describe the maximising models in both of these cases 
as well as the minimising model in the one-marginal case and relate the 
two-marginals case to the discrete time problem with two marginals.
\end{abstract}

\maketitle


\section{Introduction}
\blfootnote{We gratefully acknowledge financial support from FWF under grant P26736.

This is a version of a thesis submitted by the author in partial fulfillment of the requirements for the
degree of Master of Science at the University of Vienna.}

Determining the value of a path-dependent option (as for example Asian options)
is an important and well-developped topic in mathematical finance. The traditional
approach going back to Samuelson, Merton, Black and Scholes consists of making
a good guess at (a model of) the law of the underlying stock price which is consistent with the
information that is available on the market. If the law of the process is kept simple enough
i.e. analytically tractable, then one can calculate the fair (and arbitrage-free) price
of the option as its expected (modulo discounting with the risk-free rate) payoff
with respect to this probability law. There is a wide range of such models that have
been examined in the past and are also extensively used in the financial industry to 
calculate prices and make Value-at-Risk analyses of financial products.

A problem with this approach is the choice of the ``correct'' model. The individual
choice of the model used for pricing and risk analysis is not completely arbitrary, but
still based on strong assumptions. This lack of knowledge about the true law of the
process leads to the notion of model risk, which is generally ignored in and therefore not captured by
the traditional approach. Having no knowledge about the true process, it is almost
impossible to put your finger on the range of deviations from your proposed model.

One of the early approaches to this issue is by Hobson \cite{Ho98}.
He proposed a way to tackle this problem by instead of determining a single price
based on a single model, striving to determine the whole range of possible prices attainable by
having some (consistent) law for the process. In the aftermath of \cite{Ho98} a lot
of material was published which discussed the problem for various classes of assets.
The lookback options were discussed in \cite{Ho98}, barrier options in 
\cite{BrHoRo01a,BrHoRo01b}, basket options in \cite{dAsElG06,HoLaWa05a,HoLaWa05b} and also
volatility swaps in \cite{CaLe10,Ka11,HoKl12}. These papers generally exploit various 
solutions of the Skorokhod embedding problem which happen to have certain beneficial
properties. A summary of the approach can be found in a set of lecture notes by Hobson \cite{Ho11}.
An overview of the various useful Skorokhod embeddings is given in Obloj's survey
\cite{Ob04}.

In a more recent series of papers, a new approach to obtain the model-independent
price range has been developed using results originally devised for the optimal
transport problem as introduced by Monge \cite{Mo81} and Kantorovich \cite{Ka04,Ka58}.
Modifications of the classical duality 
results have been established by Galichon, Henry-Labordère and Touzi \cite{GaHLTo14},
Beiglböck, Henry-Labordère and Penkner \cite{BeHLPe13},
Acciaio, Beiglböck, Penkner and Schachermayer \cite{AcBePe13},
Bouchard and Nutz \cite{BoNu13} as well as Dolinsky and Soner \cite{DoSo12}. 
The method has been
applied specifically to forward start straddles in \cite{HoNe12,HoKl13b} and a general
class of exotic options depending on the stock price at two fixed points in time in
\cite{BeJu12,HLTo13}.

This article is mostly concerned with new applications of the techniques of this latter family of
papers to the case of Asian-style options. Pricing of Asian options is a
hard problem in the traditional approaches as well and can only be done 
numerically for general market settings. We will be able to give geometric
characterisations of the maximising and the minimising pricing model for certain
types of Asian options in discrete and continuous time.

First, we will establish our notation and provide a short sketch of the principal problem
of model-independent finance in the following section. Afterwards we will discuss Asian options and their various
forms and give a short overview of previous results concerning their value. Then
we will discuss Asian options which monitor the stock price at a finite number
of times and give a characterization of the models which yield the extreme values
 for the options which take an average of only two values.
Finally we will also give an overview of how these discrete time results relate
to Asian options which are defined through a continuous average over the stock price.

\section{Principal Problem of Model-Independent Finance}

We want to consider a framework in which
we can consider a space of (portfolios of) liquidly traded assets $V$. This
is just a vector space spanned by a set of random variables (which represent
the payoffs of the various assets). This framework is inspired by the one
developped by Hobson in \cite{Ho11} which can be seen as an extension of the framework introduced in
Föllmer and Schied \cite{FoSc04} and used by Cox and Oblój in \cite{CoOb11a,CoOb11b}.

To be a little more specific, we consider a 
filtered measurable space $(\Omega, \mathcal{F},(\mathbb{F}_t)_{t \in [0,T]})$ 
where it will always be sufficient to set $\Omega := C[0,T]$\footnote{In Section 5 we will
consider paths with a finite number of discontinuities as well, but demonstrate how they can be
approximated by continuous paths such that the above space is indeed sufficient.}. 
The filtration shall be the minimal right continuous filtration such that the coordinate process
$X_t(\omega) := \omega(t)$ is adapted to it.
We will interpret the coordinate process $X_t$ as the
forward (i.e.\ discounted) price of some underlying asset. While the extension of this
framework to multiple stocks is rather straightforward, we will omit it and restrict
ourselves to a single-stock setting.
 
We can treat $V$ as a linear subspace of the space of $\mathbb{F}_T$-measur-able
random variables. (The exotic pathwise options are then any $\mathbb{F}_T$-measurable
random variables not in $V$.) Depending on the market setting we want to work in, we can
assume $V$ to contain various assets. One very prominent and important type of
asset in $V$ will be vanilla call options with some strike $K \in \mathbb{R}$
($C_K := (X_T - K)_+$) which account for an integral part of many markets.
Furthermore it will generally be the case that $V$ contains the space of 
self-financing trading strategies $S \subseteq V$.

\begin{definition}[Self-Financing Trading Strategy]
If $(H_t)_{t \in [0,T]}$ is a process which is adapted to $\mathbb{F}$ and whose paths are of bounded variation,
then we call the random
variable given by 
$(H \cdot X)_T := \int_0^TH_t\, dX_t$\footnote{This integral is in the sense of 
Lebesgue-Stieltjes and thus \emph{pathwise}. The question of which classes
of processes and integrands allow for pathwise stochastic integration is a recurrent
problem in Stochastic Calculus. It has already been discussed by Bichteler \cite{Bi81}
and later on by Karandikar \cite{Ka95}. A very recent notable paper on the subject is 
Nutz \cite{Nu12}. In this thesis we can restrict ourselves to processes
of bounded (total) variation which ensures the existence of a pathwise integral
in the Riemann-Stieltjes integral with respect to a process with continuous paths due to elementary
calculus results. We will need this again in Section 5.} a \emph{self-financing trading
strategy}.
$S := \{(H\cdot X)_T : H \text{ adapted to } \mathbb{F}\}$.
\end{definition}

This just means that we can always trade in the underlying using some (regular enough)
 adapted strategy.
Note that we always work with forward prices here. Therefore we consider the risk-free
rate to be $0$, so we can borrow money to buy stocks without having any interest to
incorporate.
 
Now we can and will define a pricing functional $\mathcal{P}$ on $V$. These are the prices of the
liquidly traded assets
as they are a priori determined by the market. As all the assets in $V$ are liquidly traded
(or linear combinations of liquidly traded assets) on the market, they will have a well
determined price at a certain time $0$ (the present). As the market should be free from
arbitrage we can assume the following conditions for $\mathcal{P}$:
 
\begin{enumerate}[(P1)]
\item $\mathcal{P}(A + \lambda B) = \mathcal{P}(A) + \lambda \mathcal{P}(B)$ for $A,B \in V, \lambda \in \mathbb{R}$.
\item $A \leq B \Rightarrow \mathcal{P}(A) \leq \mathcal{P}(B)$ for $A,B \in V$.
\item $\mathcal{P}(A) = 0$ for $A \in S$.
\end{enumerate}
 
The condition (P1) (linearity) relates to the ``Law of one Price''  in finance. Similarly
(P2) is basically the ``No-Arbitrage'' condition (no gain without risk).
The last condition (P3) also follows naturally from basic considerations in financial
markets, as we can clearly obtain the payoff of a self-financing
trading strategy without needing any initial endowment, so its price should be zero.
 
The problem we now want to attack in this framework, is to put a price tag on
an additional asset $A \notin V$  which is $\mathbb{F}_T$ measurable.
We consider the space $V_A := V + (\lambda A)_{\lambda \in \mathbb{R}}$ and we want to
extend $\mathcal{P}$ to a functional $\mathcal{P}_A : V_A \to \mathbb{R}$ such that the
conditions (P1)-(P3) continue to hold. Obviously this amounts to choosing
$p_A := \mathcal{P}_A(A)$ so that we do not violate monotonicity.
 
In model-independent finance we strive to find not only one possible price, but all prices
consistent with these conditions. The set of possible prices is easily seen to be convex which 
means it is an interval, so we want to find the boundaries of this interval. This problem is a 
linear optimisation problem but varies strongly through the choice of the initial market assets 
$V$ and the asset $A$ whose price range we want to find.

\subsection{Bounding Prices by Sub-/Super-Hedging}

A necessary condition to avoid arbitrage in a price is to keep the price below
any price of a portfolio which superreplicates the derivative. The analogous lower
bound is then given by the price of a sub-replicating portfolio.

Expressed in the above framework we want to find $A_l,A_u \in V$ such that
$A_l \leq A \leq A_u$. This gives rise to the bounds 
$\mathcal{P}(A_l) \leq p_A \leq \mathcal{P}(A_u)$.

Generally we will assume that we have a generating system for $V$\footnote{In
the sense that elements of $V$ are finite linear combinations of the generating system.}. Then this system
represents the assets whose prices are given because they are actually traded in the market and
$V$ consists of the possible portfolios i.e.\ linear combinations of these assets.

\begin{definition}[Vanilla Securities]
Let $\{F_{\iota} : \iota \in I\} \subseteq V$ be a given family of random variables
such that $\{F_{\iota} : \iota \in I\}$ is a generating system of $V$, then we call this
set the set of \emph{vanilla securities}.
\end{definition}

\begin{definition}[Semi-Static Portfolio]
We call the elements in $V$ \emph{semi-static portfolios}. If we are then given an
asset $A \notin V$
semi-static portfolios $A_l := \sum_{\iota \in I} c^l_{\iota}F_{\iota}$ and
$A_u := \sum_{\iota \in I} c^u_{\iota}F_{\iota}$ with only finitely many $c^u_{\iota}$ and
$c^l_{\iota}$ being different from zero such that $A_l \leq A \leq A_u$ holds, then we
call $A_l$ a \emph{semi-static subheding portfolio} and $A_u$ a \emph{semi-static superhedging
portfolio} of $A$.
\end{definition}

\begin{remark}[Arbitrage Opportunities]
If we would
set $p_A$ such that for some $A_u \in V$ with $A \leq A_u$ we have $p_A > \mathcal{P}(A_u)$
then this would give an immediate arbitrage possibility by going short in $A$ and long in $A_u$.
\end{remark}

To avoid such a situation, we need to take the minimum over all $\mathcal{P}(A_u)$ such that 
$A \leq A_u$. This is the optimisation problem of model-independent finance. The situation for 
the lower bound is analogous.

\begin{remark}
Note that we did not formulate the inequalities for the sub- and superhedging portfolios
in an almost surely fashion because we do not prescribe an underlying model. 
\end{remark}

\subsection{Bounding Prices by Models}

Another way of calculating the range of possible prices is by optimising the prices obtained
from the traditional method.
The traditional method for finding a price, described at the beginning, is to prescribe some consistent law.

\begin{definition}[Consistent Martingale Measure]
A measure $\mathbb{Q}$ such that for any vanilla security $F_{\iota}$ we have 
$\mathcal{P}(F_{\iota}) = E_{\mathbb{Q}}[F_{\iota}]$ is called a \emph{consistent martingale measure}.
We will denote the set of consistent martingale measures by $\mathcal{M}(\mathcal{P})$.
\end{definition}

Given some consistent martingale measure $\mathbb{Q}$ one can then calculate the price of the 
additional asset $A$ as $E_{\mathbb{Q}}[A]$.

\begin{remark}
As already mentioned we want to ignore discounting here. So one can either consider the
risk-free rate to be zero or consider the given payoffs of the vanilla securities and the
additional (exotic) asset $A$ to be already discounted the way they are given. This
did not play a role in the above valuation method through hedging because discounting
is always monotone and therefore does not change whether a semi-static portfolio is a
superhedge or not.
\end{remark}

A price $p_A$ given in this fashion automatically fulfills the consistency conditions by the properties
of the expected value. As above we can maximise over all possible laws $\mathbb{Q}$ to obtain a possible
upper bound on the interval of consistent prices for $A$. 

Note that there is no a priori reason why there has to exist any law $\mathbb{Q}$ such that 
existing prices are consistent with it in an arbitrary market setting $V$ (the fact that prices
are expected values under some law of the price process is already an assumption). Just as well, 
there is no a priori reason why there cannot be an arbitrage free price for the asset 
which cannot be achieved by a law that gives rise to the market prices.

If we want to know that the optimal bound from this optimisation problem really is the optimal
upper bound for the interval we want to find, we need to prove strong duality for the two 
optimisations we just lined out. This means we need to check that
\[\sup_{\mathbb{Q} \in \mathcal{M}(\mathcal{P})} E_{\mathbb{Q}}[A] = \inf_{A\leq A_u \in V} \mathcal{P}(A_u) \text{.}\]

This duality has been established in various market settings $(V,\mathcal{P})$ for instance by \cite{BeHLPe13}
and \cite{GaHLTo14}. Note that for any $\mathbb{Q}\in\mathcal{M}(\mathcal{P})$ and
$A_u \in V$ with $A \leq A_u$ we obviously have $E_{\mathbb{Q}}[A] \leq E_{\mathbb{Q}}[A_u] = \mathcal{P}(A_u)$.
Hence the inequality 
$\sup_{\mathbb{Q} \in \mathcal{M}(\mathcal{P})} E_{\mathbb{Q}}[A] \leq \inf_{A\leq A_u \in V} \mathcal{P}(A_u)$ is
trivial.
We will get back to the duality result of \cite{BeHLPe13} in section 4.

\subsection{Market Restrictions}

To get a better idea on the set of measures we will usually find in $\mathcal{M}(\mathcal{P})$ we want
to discuss a few common restrictions here.

\begin{lemma}
Consider a framework as above and consider some measure $\mathbb{Q}$ on the
measurable space $(\Omega,\mathcal{F})$ such that the functional 
$\mathcal{P}(A) := E_{\mathbb{Q}}[A]$ on $V$ fulfills (P1)-(P3). Then the coordinate
process $(X_t)_t$ is a martingale with respect to $\mathbb{Q}$.
\end{lemma}

\begin{proof}
We know from the discussion so far that $V$ contains all
self-financing strategies, and they have $0$ cost by (P3).
Under this condition, consider some $0\leq t_1 \leq t_2 \leq T$ and
some arbitrary function $h \in C_b(\mathbb{R})$ and observe that
\begin{align*}
E_{\mathbb{Q}}[h(X_{t_1})(X_{t_2}-X_{t_1})] 
	&= E_{\mathbb{Q}}\left[\int_0^Th(X_{t_1})\mathds{1}_{t_1 \leq t \leq t_2}dX_t\right] \\
	&= \mathcal{P}\left(\int_0^Th(X_{t_1})\mathds{1}_{t_1 \leq t \leq t_2}dX_t\right) = 0 \text{.}
\end{align*}
\end{proof}
 
Now suppose that prices of vanilla calls for some maturity $0 < t \leq T$ and all possible strikes 
$K \in \mathbb{R}$ are known. It has been observed by Breeden and Litzenberger \cite{BrLi78} that 
this determines the marginal at time $t$, i.e.\ the distribution of $X_t$ (the value of $X$ at the specific time $t$, not
the whole process) with respect to $\mathbb{Q}$.

This can be motivated by considering the function $C(K) := E_{\mathbb{Q}}[(X_t-K)_+] = \mathcal{P}((X_t-K)_+) 
= \int (x-k)_+ d\mu_(x)$ where $\mu = \mathrm{Law}_{\mathbb{Q}}(X_t)$.
This means that we assume that the market values European call options with maturity $T$ according
to some law $\mathbb{Q}$ and we know the price it obtains that way.
The first derivative (in the distribution sense) of this function is then given by $C'(K) = E_{\mu}[-\mathds{1}_{K < X_t}]
=-(1-\mu[X_t \leq K])$. This allows us to recover the distribution function of $\mu \sim X_t{}_\#\mathbb{Q}$ from the
call prices.

\section{A Short Overview of Asian Option Pricing}

Eventually we strive to establish (or rather characterise) sharp model-independent price bounds for 
Asian options.

An Asian option is a bet on the average value of an option over a fixed time frame between
initiation and expiration of the option.
By considering the average instead of the price of the stock at some specific point in
time, one obtains an asset which has less volatility compared to the stock itself. 
This usually makes Asian options cheaper than regular European-style 
options\footnote{In Section 5 we will actually show that Asian options can be superhedged with a
portfolio consisting of a European option with the same maturity and a self-financing
trading strategy.}.

Depending on the precise form of the contract, these options occur in various 
different forms.
The main distinguishing feature is the granularity and type of the averaging. We will only
be interested in the classical arithmetic mean for a continuously or discretely monitored 
stock price
\[A^c := \frac{1}{T}\int_0^TX_t \, dt\text{, } \qquad A^d := \frac{1}{N}\sum_{i=1}^N X_{t_i}.\]
Here we have $0 \leq t_1 < \dots < t_N \leq T$. Alternatively one could consider
geometric averaging or averaging of the logarithmic values.

On these averages the usual bet comes in the form of a ``hockey-stick''-function. So 
one can consider Asian puts or calls as $(A^{c,d} - K)_+$ and $(K-A^{c,d})_+$ respectively.
If we can trade in a continuum of various strikes, we have the possibility to synthesize
arbitrary convex functions of the average, so a general Asian option will have the payoff
$\phi(A^{c,d})$ for some convex function $\phi$.

The pricing of Asian options is especially interesting as a topic for research, as there is not even a closed form
solution for the price if we assume a standard Black-Scholes model. 
See for instance the survey article by Boyle and Potapchik \cite{BoPo08} for traditional approaches to pricing
Asian options.

Model independent bounds for Asian options of various types have also been an active topic of discussion. An
early work on this topic is an article by Dhaene et al.\ \cite{DhDeGo02}. More recently Albrecher et al.\ \cite{AlMaSc08}
proposed a model independent lower bound.
Improvements under certain strong conditions to the market have been discussed by Deelstra et al.\ in \cite{DeRaVa14}.
We can follow their approach for finding a lower bound in our framework very easily by exploiting Jensen's inequality 
for convex functions to show for any joint law of the price process that the following holds:
\begin{align*}
E\left[\phi\left(\frac{1}{N}\sum_{i=1}^N X_{t_i}\right)\right] &= E\left[E\left[\left.\phi\left(\frac{1}{N}\sum_{i=1}^N X_{t_i}\right)\right|\mathbb{F}_{t_1}\right]\right] \\
&\geq E\left[\phi\left(\frac{1}{N}\sum_{i=1}^N E[X_{t_i}|\mathbb{F}_{t_1}]\right)\right] = E[\phi(X_{t_1})] \text{.}
\end{align*}
So we find a lower bound for the price of an Asian option by the price of a European option with maturity
$t_1$ which is the first sample point of the averaging. (Provided we know this price, which is assumed in \cite{DeRaVa14}.)

These bounds are not shown to be sharp in \cite{AlMaSc08} or \cite{DeRaVa14}. This means we only know that 
$\inf_{\mathbb{Q} \in \mathcal{M}(\mathcal{P})} E_{\mathbb{Q}}[A] \geq E[B]$ where $A$ is the 
Asian option and $B$ the European option with maturity $t_1$. It can be shown that the bound is sharp, if
we only know the marginals at time $t_1$. This can be achieved in similarly to the methods we use in Section 5,
but we will not pursue this further here. In \cite{DeRaVa14} it is assumed though that European
options with maturities $t_1, \dots,t_n$ are traded with well-known prices. In this case the following simple example illustrates
that in general the bound is not sharp.

\begin{example}
Consider a forward price with marginals $\mu_{0} := \delta_1$, $\mu_{t_1} := \delta_1$ and $\mu_{t_2}:=\frac{1}{2}(\delta_0+\delta_2)$ as well as the Asian option with discrete averaging $(\frac{1}{2}\sum_{i=1}^2 X_{t_i}-1)_+$.
The lower price bound obtained in this way is then obviously $0$, while the only possible true forward
price obtainable by choosing a consistent martingale measure is $\frac{1}{4}$.
\end{example}

\section{Bounds on Discrete Time Options}

\label{secdis}
We will now try to improve on this bound in our setting by finding the actual model $\mathbb{Q}$
which minimises the functional above.
Eventually we have to restrict ourselves to an Asian option which is written 
on the average of the stock price at two points of time $0 < t_1 < t_2$.
So we want to
give a price for a derivative with payoff $\phi\left(\frac{X_{t_1}+X_{t_2}}{2}\right)$ where $\phi$ is a convex
function. We will want this price to be optimal in the above given sense, using only the information 
about the call prices for vanilla
options with maturity $t_1$ and $t_2$ or equivalently the laws of $X_{t_1}$ and $X_{t_2}$.

The mathematical objects we need to consider will be given in the following
\begin{definition}
Consider two measures $\mu$ and $\nu$. The set of \emph{martingale couplings}  or \emph{martingale transports} 
for these measures is given by
\begin{align*} 
\mathcal{M}(\mu,\nu) &:= \{\rho \in P(\mathbb{R}^2) : \mathrm{proj}^x\#\rho = \mu,\mathrm{proj}^y\#\rho = \nu,\\
&\qquad \rho \text{ is a martingale measure}\} \text{.}
\end{align*}
where $P(\mathbb{R}^2)$ denotes the probability measures on $\mathbb{R}^2$.

For the maximal and minimal cost achievable by \emph{transporting}
$\mu$ to $\nu$ along a \emph{martingale} we set
\begin{align*}
\bar{D} &:= \sup_{\pi \in \mathcal{M}(\mu,\nu)} \int |x+y| \, d\pi(x,y) \text{;}\\
\underaccent{\bar}{D} &:= \inf_{\pi \in \mathcal{M}(\mu,\nu)} \int |x+y| \, d\pi(x,y) \text{.}
\end{align*}
\end{definition}
Using this definition we will now formulate the following theorem to characterise the optimising model.
\begin{theorem}
\label{mainthm}
Consider two measures $\mu$ and $\nu$ in convex order such that $\mu$ is
continuous and $\supp(\mu) \subseteq \mathbb{R}^{+}$. Furthermore we require $\int |y|^3 \, d\nu < \infty$.

We can find $\bar{\pi},\underaccent{\bar}{\pi} \in \mathcal{M}(\mu,\nu)$ such that
\[\bar{D} = \int |x+y| \, d\bar{\pi}(x,y)  \qquad \underaccent{\bar}{D} = \int |x+y| \, d\underaccent{\bar}{\pi}(x,y) \]
holds. 

Furthermore $\bar{\pi}$ can be chosen such that it is concentrated on the graphs of two real Borel measurable
functions $\bar{T}_u:\mathbb{R}^+ \to \mathbb{R}^+$ and $\bar{T}_l: \mathbb{R}^+ \to \mathbb{R}$. $\bar{T}_u$ is non-increasing.

Similarly $\underaccent{\bar}{\pi}$ can be chosen such that is is concentrated on the graphs of two real
functions $\underaccent{\bar}{T}_u:\mathbb{R}^+ \to \mathbb{R}^+$ and $\underaccent{\bar}{T}_l: \mathbb{R}^+ \to \mathbb{R}$ and the secondary diagonal $-\mathrm{Id}$.
$\underaccent{\bar}{T}_u$ is non-decreasing.
\end{theorem}

\begin{corollary}
If we have a setup as above, where the support of $\mu$ is not restricted to the positive
half of $\mathbb{R}$ then there exist maximising and minimising martingale measures in
$\mathcal{M}(\mu,\nu)$ whose support is of the following form (respectively):
\begin{itemize}
\item On the right half-plane of $\mathbb{R}^2$, the support of the optimiser looks as described
in Theorem \ref{mainthm}.
\item On the left half-plane, the support is of the form described in Theorem \ref{mainthm} rotated around
the center of the plane by $180^{\circ}$.
\end{itemize}
\end{corollary}

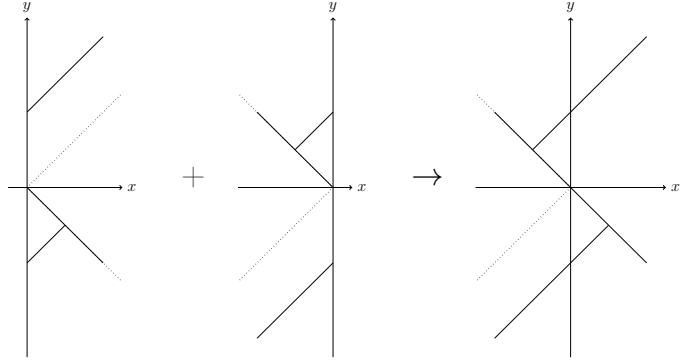
\begin{figure}
\begin{center}
\begin{tabular}{ccccc}
\parbox[c]{2cm}{
\scalebox{0.5}{
\begin{tikzpicture}
\draw[->] (-0.5,0) -- (2.5,0) node[right] {$x$};
\draw[->] (0,-4.5) -- (0,4.5) node[above] {$y$};
\draw[dotted] (0,0) -- (2.5,2.5);
\draw[dotted] (0,0) -- (2.5,-2.5);
\draw[domain=0:2,variable=\x] plot ({\x},{-\x});
\draw[domain=0:2,variable=\x] plot ({\x},{\x+2});
\draw[domain=0:1,variable=\x] plot ({\x},{\x-2});
\end{tikzpicture}}} & + &
\parbox[c]{2cm}{
\scalebox{0.5}{
\begin{tikzpicture}
\draw[->] (-2.5,0) -- (0.5,0) node[right] {$x$};
\draw[->] (0,-4.5) -- (0,4.5) node[above] {$y$};
\draw[dotted] (0,0) -- (-2.5,2.5);
\draw[dotted] (0,0) -- (-2.5,-2.5);
\draw (0,2) -- (-1,1);
\draw (-2,-4) -- (0,-2);
\draw (0,0) -- (-2,2);
\end{tikzpicture}}}
& $\rightarrow$ &
\parbox[c]{3cm}{
\scalebox{0.5}{
\begin{tikzpicture}
\draw[->] (-2.5,0) -- (2.5,0) node[right] {$x$};
\draw[->] (0,-4.5) -- (0,4.5) node[above] {$y$};
\draw[dotted] (0,0) -- (-2.5,2.5);
\draw[dotted] (0,0) -- (-2.5,-2.5);
\draw (2,4) -- (-1,1);
\draw (-2,-4) -- (1,-1);
\draw (2,-2) -- (-2,2);
\end{tikzpicture}}}
\end{tabular}
\end{center}
\caption{Concentration sets of minimising transport plans with purely positive and negative $\mu$'s as well as their mixture}
\end{figure}

Before we start to explain the significance of this theorem, we want to point out a few
important details about it.

One might notice that we restrict ourselves here to a single payoff function, namely the
absolute value, instead of all possible convex functions. While it might be favourable to
cover arbitrary convex functions, this seems a little too hard a task at this point. The absolute
value does cover the most important type of convex functions used for options, which is the
``hockey-stick''-function bet. Note that we have
\[\left(\tfrac{x+y}{2} - K\right)_+ = \left(\tfrac{x-K}{2}+\tfrac{y-K}{2}\right)_+ 
= \tfrac{1}{2}\left|\tfrac{x-K}{2}+\tfrac{y-K}{2}\right|+\tfrac{x+y}{4} +\tfrac{K}{2} \text{.}\]
So, as the value of $\int (x+y)/2+K\,d\pi(x,y)$ is constant over all measures $\pi \in \mathcal{M}(\mu,\nu)$,
it is equivalent to the original problem to maximise the total value for
$\int |\tilde{x}+\tilde{y}|\,d\tilde{\pi}(\tilde{x},\tilde{y})$ over $\tilde{\pi} \in \mathcal{M}(\tilde{\mu},\tilde{\nu})$ where we set
$\tilde{\mu} := (2x+K)_\#\mu$ and $\tilde{\nu} := (2x+K)_\#\nu$ respectively.
This transformation makes it necessary to deal with measures $\mu$ and $\nu$ with support
in all of $\mathbb{R}$ even if one is only interested in the financial applications of the theorem
where $\mu$ and $\nu$ are usually only supported on $\mathbb{R}^+$.

The optimisers (maximiser or minimiser) of this optimisation problem are not unique. This
can be easily seen by considering measures $\mu$ and $\nu$ such that 
$\supp(\mu),\supp(\nu)\subseteq \mathbb{R}^+$. In this case any measure in $\mathcal{M}(\mu,\nu)$
will lead to the same value and is therefore a maximiser and a minimiser at the same time. This
problem of course persists in the case of measures which are not restricted to the positive
half-axis.

Furthermore note that this theorem belongs to the theory of optimal martingale transport
which is optimal transport with an additional (linear) constraint, namely that the transport
is a martingale measure. In this theory, the function that we call a payoff function
(here $c(x,y) = |x+y|$) is usually referred to as a cost function. We will use these terms
alternatingly in the context of two-step martingale measures.

We will defer the explanation of this theorem to a later point and proceed first with some
 further specializations of the described setup and
previous results that can be applied to obtain conclusions for Asian options of 
the above type.

First, we consider the class of path-dependent options which only depend on a finite
number of values of the path. So we consider payoffs of the form
\[\Phi((X_t)_{t\in [0,T]}) := \bar{\Phi}(X_{t_1},\dots,X_{t_n}) \]
with $0 \leq t_1 < \dots < t_n \leq T$. (Obviously a discretely monitored Asian option
is of this type).
Without loss of generality, we can even set $t_i := i$ and $T=n$, which we will
clarify later. When $(X_t)_t$ is a martingale in continuous
time with respect to some measure $\mathbb{Q}$, then $X_1, \dots, X_n$ is 
obviously a discrete time martingale with respect to $\mathbb{Q}$. We furthermore assume that
the marginal laws of the $X_i$ are given as discussed before. Furthermore, if we fix a discrete time
martingale $Y_1, \dots, Y_n$, we can always construct a martingale with the correct marginals
by setting $X_t := Y_i$ for $i-1 < t \leq i$. So, if we maximise the above payoff over all discrete time
martingales, we get the same as if we maximise over all continuous time martingales. We now need
that this maximisation indeed gives us the optimal robust bound for this payoff. This is the
aforementioned result about strong duality which has been shown in \cite{BeHLPe13}. 
We will discuss it in detail in the following subsection.

\subsection{Strong Duality in Discrete Time}

The minimal upper bound on the price of a derivate with payoff $\Phi(X_1, \dots, X_n)$
achievable by a consistent law
is the minimal price of a super hedge of that option. For this we need to consider the set of feasible
superhedges.

In the framework described in Section 2, we let the space $V$ be spanned by the self-financing trading
strategies and the european call options, i.e.\ derivatives with payoffs $(X_i-K)_+$ for some $i=1,\dots,n$
and some real number $K$.

This means that the semi-static portfolios are of the form
\[H_{(\phi_i)_i,(h_j)_j} := \sum_{i=1}^n \phi_i(x_i) + \sum_{j=1}^{n-1} h_j(x_1, \dots, x_j)(x_{j+1}-x_j) \]
with $\phi_i \in L^1(\mu_i)$ convex, and $h_j: \mathbb{R}^j \to \mathbb{R}$ bounded measurable. Where $\mu_i$ is the probability
law of $X_i$ which is uniquely determined through the price
functional $\mathcal{P}$ by the condition $E_{\mu_i}[(X_i-K)_+] = \mathcal{P}((X_i-K)_+)$ for $K \in \mathbb{R}$ as outlined before.

For future use we define
\begin{definition}[Superhedging Strategies]
Given some payoff $\Phi: \mathbb{R}^n \to \mathbb{R}$ we call
\[ \mathcal{H}(\Phi) := \{((\phi_i)_{i=1,\dots,n},(h_j)_{j=1,\dots,n-1}) :\Phi(x_1,\dots,x_n) \leq H_{(\phi_i)_i,(h_j)_j}\}\]
the set of feasible \emph{superhedging} strategies.
\end{definition}
The price of any $H_{(\phi_i)_i,(h_j)_j}$ is then given by 
\[ \mathcal{P}(H_{\phi_i,h_j}) = \sum_{i=1}^n \mathcal{P}(\phi_i(X_i)) + \sum_{j=1}^{n-1}\mathcal{P}(h_j\cdot(X_{j+1}-X_j))
= \sum_{i=1}^n \mathcal{P}(\phi_i(X_i))\]
where we used conditions (P1)-(P3) from Section 2.
The optimal robust bound is then formally given by
\[\bar{P} := \inf\left\{\sum_{i=1}^n E_{\mu_i}[\phi_i] : \exists (h_j)_{j=1,\dots,n-1} \text{ s.t. } H_{(\phi_i)_i,(h_j)_j} \in \mathcal{H}(\Phi)\right\}\text{.}\]

The above problem has the inconvenient drawback that the set of superhedges is not compact and therefore
the infimum is not always a minimum, which has been shown in \cite{BeHLPe13}. 
This can be avoided by considering the dual problem instead.

We already know from Section 2 that any \emph{consistent law} $\mathbb{Q}$ has to be chosen such that 
the process $X_i$ is a martingale i.e.\ $E_{\mathbb{Q}}[X_i|X_1,\dots,X_{i-1}]$ $= X_{i-1}$ for $2 \leq i \leq n$.
Furthermore, the distribution of $X_i$ with respect to $\mathbb{Q}$ has to be $\mu_i$ as described above.
Clearly only the distribution of $(X_1,\dots,X_n)$ is of interest for the problem, so we will work with measures
on $\mathbb{R}^n$ instead of the original (abstract) $\Omega$.
Let $\mathcal{M}(\mu_1,\dots,\mu_n)$ denote the set of measures on $\mathbb{R}^n$ which have
marginals $\mu_i$ and are marginal laws. $\mathcal{M}(\mu_1,\dots,\mu_n)$ is not empty if and only if the
marginals $\mu_i$ are in \emph{convex order}. This result has already been observed by Strassen in \cite{St65}.
A proof for this fact is also outlined in \cite{BeJu12}.

Now we can consider the dual optimisation problem
\[\bar{D} := \sup\{E_{\mathbb{Q}}[\Phi(X_1,\dots,X_n)] : \mathbb{Q} \in \mathcal{M}(\mu_1,\dots,\mu_n)\} \text{.}\]
For any $H_{(\phi_i)_i,(h_j)_j} \in \mathcal{H}(\Phi)$ and $\mathbb{Q} \in \mathcal{M}(\mu_1,\dots,\mu_n)$ we
obviously have $\mathcal{P}(H_{(\phi_i)_i,(h_j)_j}) = \sum_{i=1}^n E_{\mu_i}[\phi_i(X_i)] = E_{\mathbb{Q}}[H_{(\phi_i)_i,(h_j)_j}]$
due to the martingale property of $\mathbb{Q}$ and because
$\mathcal{P}(\phi_i(X_i)) = E_{\mu_i}[\phi_i(X_i)] = E_{\mathbb{Q}}[\phi_i(X_i)]$. Furthermore $E_{\mathbb{Q}}[H_{(\phi_i)_i,(h_j)_j}] \geq E_{\mathbb{Q}}[\Phi]$.
holds because of the monotonicity of the expected value. This shows that $\bar{P} \geq \bar{D}$ always holds
(if the dual problem is infeasible it holds even trivially). As announced before, strong duality has been
shown in \cite{BeHLPe13} which is summarised in the following theorem:

\begin{theorem}
Assume that $\mu_1, \dots, \mu_n$ are Borel probability measures on $\mathbb{R}$ such that
$\mathcal{M}(\mu_1,\dots,\mu_n)$ is non-empty. Let $\Phi : \mathbb{R}^n \to (-\infty,\infty]$
be a lower semi-continuous function such that
\[\Phi(x_1,\dots,x_n) \leq K \cdot(1+|x_1|+\dots + |x_n|) \]
on $\mathbb{R}^n$ for some constant $K$. Then there is no duality gap, i.e. $\bar{P}=\bar{D}$. Moreover, the dual
value $\bar{D}$ is attained, i.e. there exists a martingale measure $\mathbb{Q} \in \mathcal{M}(\mu_1,\dots,\mu_n)$
such that $\bar{D} = E_{\mathbb{Q}}[\Phi]$.
\end{theorem}

The last assertion is due to weak compactness of $\mathcal{M}(\mu_1,\dots,\mu_n)$. 

\subsection{Geometric Description of the Dual Optimiser}

The dual problem from the last section
is similar to results obtained for the classical optimal transport problem on $\mathbb{R}^n$ as extensively
covered in \cite{Vi03} and \cite{Vi09}. A very common feature in optimal transport is that
optimal transport plans are often concentrated on sets which are geometrically simple. The most famous of
these results is Brenier's theorem which describes the geometry of optimal transport plans for the quadratic
cost function. It asserts that these transport plans are so called monotone transport plans. This means
that they are given by the gradient of a convex function on $\mathbb{R}^n$. For one dimensional
measures which we will consider from now on in the martingale case, this amounts to saying that the transport plan
is concentrated on the graph of a non-decreasing function.

This inspires to look for geometric properties of optimal martingale couplings (i.e. transport plans). Due to the
similarity of the problem,
we will call the elements of $\mathcal{M}(\mu,\nu)$ (with $\mu$ and $\nu$ in convex order) as defined above \emph{martingale transport plans}.  This problem
has been considered by various authors in the case of options depending on only two points in time 
$\Phi((X_t)_{t\in [0,2]}) := c(X_1,X_2)$.
Hobson and Neuberger \cite{HoNe12} and Hobson and Klimmek \cite{HoKl12} describe the geometry of optimisers for
the forward start straddle $c(x,y) := |x-y|$. Beiglböck and Juillet \cite{BeJu12} and Henry-Labordère and Touzi
\cite{HLTo13} consider maximisers and minimisers of options with a payoff of the form $c(x,y)$ with 
$c_{xyy} > 0$ and $c_{xyy} < 0$ respectively. 
A martingale transport plan cannot be concentrated on a single
graph like in the non-martingale problem unless the two measures agree, in which case the identical transport is the only feasible martingale transport plan. In the cases previously considered, 
they are usually concentrated on two or three graphs which seems to
be the natural analogon in the martingale case.

A widely used tool to derive the geometry of a two dimensional optimal transport plan in the classical optimal transport
problem is the notion of cyclical monotonicity which we will not discuss further here. There is a similar notion
derived in \cite{BeJu12}, which can be used to desribe optimal martingale transport plans. It is called a
\emph{variational lemma}.

\begin{lemma}[Variational Lemma]
\label{varlem}
Let $\mu$ and $\nu$ be probability measures in convex order and $c:\mathbb{R}^2 \to \mathbb{R}$ a Borel
measurable cost function which satisfies $c(x,y) \geq a(x) + b(y)$ for integrable (with respect to $\mu$ and $\nu$,
respectively) functions $a$ and $b$. Assume that $\pi \in \mathcal{M}(\mu,\nu)$ is a
minimising martingale transport plan which leads to finite costs. Then there exists a Borel set $\Gamma$ with
$\pi(\Gamma) = 1$ such that the following holds:

If $\alpha$ is a measure on $\mathbb{R}^2$ with finite support such that $\supp{\alpha} \subseteq \Gamma$ then
we have $\int c \,d\alpha \leq \int c\, d\alpha'$ for every measure $\alpha'$ such that
\begin{enumerate}[(i)]
\item $\alpha'$ has the same marginals as $\alpha$, and
\item $\int y \, d\alpha_x(y) = \int y\, d\alpha'_x(y)$ for $(\mathrm{proj}^x\#\alpha)$-a.e. $x$
\end{enumerate}
where $(\alpha_x)_{x\in \mathbb{R}}$ denotes the disintegration of $\alpha$ with respect to
$\mathrm{proj}^x\#\alpha$.
\end{lemma}

Actually we need for the proof of Theorem \ref{mainthm}
a generalisation of this lemma, which we will formulate in the following. It follows by a slight variation
of the proof for Lemma \ref{varlem} given in \cite{BeJu12}. Also it follows directly from the results
in \cite{BeGr14}.

\begin{lemma}
\label{modlem}
Let $\mu$ and $\nu$ be probability measures in convex order and $c,c':\mathbb{R}^2 \to \mathbb{R}$ Borel
measurable cost functions such that $c$ satisfies the condition in Lemma \ref{varlem} and
$c'$ satisfies the analogous condition that $c'(x,y) \leq f(x)+g(y)$ for integrable functions $f$ and $g$.
Assume that $\pi \in \mathcal{M}(\mu,\nu)$ is a
minimising martingale transport plan which leads to finite costs and maximises
$c'$ among all minimising martingale transport plans where $\int c' d\pi > -\infty$ holds. Then there exists a Borel set $\Gamma$ 
(the \emph{monotonicity set} of $\pi$) with
$\pi(\Gamma) = 1$ such that the following holds:

If $\alpha$ is a measure on $\mathbb{R}^2$ with finite support such that $\supp{\alpha} \subseteq \Gamma$ then
we have $\int c \,d\alpha \leq \int c\, d\alpha'$ and
$\int c \,d\alpha = \int c\, d\alpha'$ implies $\int c' \,d\alpha \geq \int c'\, d\alpha'$ for 
every measure $\alpha'$ as in Lemma \ref{varlem}. 
\end{lemma}

The intuitive principle is that if a transport plan was not optimal, it could be improved by 
\emph{finite} rearrangements of mass. The conditions on $c$ and $c'$ make sure that the
target values are finite.

In addition to the important Lemma \ref{varlem} we will need another auxiliary lemma which is 
also taken from \cite{BeJu12}.

\begin{lemma}
\label{comblem}
Let $k$ be a positive integer and $\Gamma \subseteq \mathbb{R}^2$. Assume also that there are
uncountably many $a \in \mathbb{R}$ satisfying $|\Gamma_a| \geq k$ where
$\Gamma_a := \{b : (a,b) \in \Gamma\}$.

There exist $a$ and $b_1 < \dots < b_k \in \Gamma_a$ such that for every $\varepsilon > 0$ one may find $a' > a$
and $b_1' < \dots < b_k' \in \Gamma_{a'}$ with
\[ \max(|a-a'|,|b_1-b_1'|, \dots, |b_k-b_k'|) < \varepsilon \text{.}\]
Moreover one may also find $a''<a$ and $b_1'' < \dots < b_k'' \in \Gamma_{a''}$ with
\[ \max(|a-a''|,|b_1-b_1''|, \dots, |b_k-b_k''|) < \varepsilon \text{.}\]
\end{lemma}

\subsection{Minimising Transport Layout}

We are now equipped with the tools we need to finally prove the structure of the optimisers proposed
in Theorem \ref{mainthm}. Due to its complexity we will split the proof in two parts. We will first prove
the structure of the minimiser and proceed with the maximiser in the next subsection.

The general idea of the proof is inspired by the proof of the optimising structure for the cost function $|x-y|$
as it is proposed by \cite{BeJu12}. The optimisers for this cost function are unique though which is not
the case here which adds to the complexity of the proof. To avoid non-uniqueness we propose a
secondary optimisation problem. Among all the minimising transport plans for the cost function
$c(x,y):=|x+y|$ we are looking for the transport plan which maximises the value of
$\int c(x,y)y^2 \, d\pi(x,y)$. Using this secondary optimisation problem one obtains exactly
the proposed form of Theorem \ref{mainthm}.

Note that the condition stated in Theorem \ref{mainthm} that the third moment of $\nu$ is finite corresponds
to the conditions on $c$ and $c'$ in Lemma \ref{modlem}. It is straightforward to check that
$0 \leq |x+y| \leq |x|+|y|$ as well as
$0 \leq |x+y|y^2 \leq |x|y^2 + |y|^3 \leq |x|^3 + 2 |y|^3$ holds. Using that $\mu$ also has finite third moment
because $\mu \preceq \nu$, we have that all the costs appearing here are finite.

We will outsource parts of the proof to one lemma and two easy corollaries.

\begin{lemma}
\label{minlem}
Let $\pi$ be a minimising martingale transport plan which is a maximiser for the
secondary optimisation problem for given marginal distributions $\mu$ and $\nu$ in
convex order and let $\Gamma$ be 
its monotonicity set as given by Lemma \ref{modlem}. 
Let $y^- < y < y^+$ and some $\lambda \in [0,1]$
such that
$y = (1-\lambda)y^- + \lambda y^+$ holds. Consider the functions
\begin{align*}
f(t) &= (1-\lambda)|t+y^-|+\lambda|t+y^+| - |t+y| \text{ and} \\
g(t) &= (1-\lambda)|t+y^-|(y^-)^2 + \lambda|t+y^+|(y^+)^2 - |t+y|y^2 \text{.}
\end{align*}

\begin{enumerate}[(i)]
\item There are no $x,x'$ such that $(x,y^-),(x,y^+),(x',y) \in \Gamma$ and $f(x) > f(x')$.

\item There are no $x,x'$ such that $(x,y^-),(x,y^+),(x',y) \in \Gamma$ and $f(x) = f(x')$ with
$g(x)<g(x')$.
\end{enumerate}
\end{lemma}

\begin{proof}
Obviously it follows from $f(x) > f(x')$ that we have 
\[(1-\lambda)|x+y^-|+\lambda|x+y^+| + |x'+y| > (1-\lambda)|x'+y^-|+\lambda|x'+y^+| + |x+y| \text{.}\]
So if we consider the finite measure $\alpha := \lambda \delta_{(x,y^-)} + (1-\lambda)\delta_{(x,y^+)} + \delta_{(x',y)}$, then $\alpha' :=\lambda \delta_{(x',y^-)} + (1-\lambda)\delta_{(x',y^+)} + \delta_{(x,y)}$ is a competitor of $\alpha$ such that 
$\int c \,d\alpha > \int c \,d\alpha'$. Hence we have a contradiction to Lemma \ref{varlem}.

The same argument works for the second part using $f$ and $g$.
\end{proof}

\begin{corollary}
\label{mincor}
Consider some minimising (and secondary maximising) martingale transport $\pi$ with monotonicity set $\Gamma$ 
for marginals $\mu \preceq \nu$ with $\supp(\mu) \subseteq \mathbb{R}^{+}$ and 
$(x,y^+),(x,y^-),(x',y') \in \Gamma$ such that $y^- < -x < y^+$ and $y^- < y' < y^+$. Then it cannot be
that 
\begin{enumerate}[(i)]
\item $-x < y'$ and $x < x'$ or
\item $y' < -x$ and $x' < x$.
\end{enumerate}
\end{corollary}

\begin{proof}
We draw paradigmatic graphs for the function $f$ from the above Lemma which show that in these cases
it holds that $f(x) > f(x')$ which implies the result by the above Lemma \ref{minlem}:
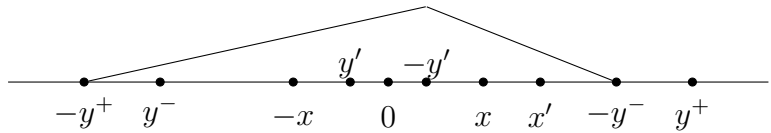
\begin{center}
\begin{tikzpicture}[every node/.style={draw,circle,fill=black,inner sep=0pt,minimum size=3pt}]
\node[label={[yshift=-20pt]$0$}] at (0,0) {};
\node[label={[yshift=-20pt]$x$}] (x) at (1.25,0) {};
\node[label={[yshift=-24pt]$-x$}] (-x) at (-1.25,0) {};
\node[label={[yshift=-26pt]$-y^-$}] (-y-) at (3,0) {};
\node[label={[yshift=-22pt]$y^-$}] (y-) at (-3,0) {};
\node[label={[yshift=-26pt]$-y^+$}] (-y+) at (-4,0) {};
\node[label={[yshift=-23pt]$y^+$}] (y+) at (4,0) {};
\node[label={[yshift=-21pt]$x'$}] (x') at (2,0) {};
\node[label={[yshift=-6pt]$-y'$}] (-y') at (0.5,0) {};
\node[label={[yshift=-2pt]$y'$}] (y') at (-0.5,0) {};
\node[draw=none,fill=none,minimum size=0pt,inner sep=0pt] (t) at (0.5,1) {};
\draw (-y+) -- (t);
\draw (-y-) -- (t);
\draw (-5,0) -- (5,0);
\end{tikzpicture}
\end{center}

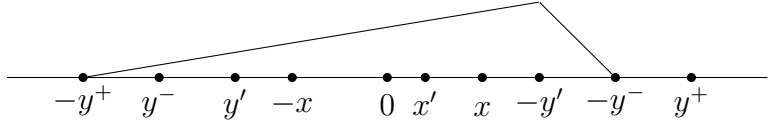
\begin{center}
\begin{tikzpicture}[every node/.style={draw,circle,fill=black,inner sep=0pt,minimum size=3pt}]
\node[label={[yshift=-17pt]$0$}] at (0,0) {};
\node[label={[yshift=-18pt]$x$}] (x) at (1.25,0) {};
\node[label={[yshift=-21pt]$-x$}] (-x) at (-1.25,0) {};
\node[label={[yshift=-23pt]$-y^-$}] (-y-) at (3,0) {};
\node[label={[yshift=-20pt]$y^-$}] (y-) at (-3,0) {};
\node[label={[yshift=-23pt]$-y^+$}] (-y+) at (-4,0) {};
\node[label={[yshift=-20pt]$y^+$}] (y+) at (4,0) {};
\node[label={[yshift=-18pt]$x'$}] (x') at (0.5,0) {};
\node[label={[yshift=-22pt]$-y'$}] (-y') at (2,0) {};
\node[label={[yshift=-19pt]$y'$}] (y') at (-2,0) {};
\node[draw=none,fill=none,minimum size=0pt,inner sep=0pt] (t) at (2,1) {};
\draw (-y+) -- (t);
\draw (-y-) -- (t);
\draw (-5,0) -- (5,0);
\end{tikzpicture}
\end{center}
\end{proof}

\begin{corollary}
\label{mincorb}
Consider some minimising (and secondary maximising) martingale transport $\pi$ with monotonicity set $\Gamma$ for marginals $\mu \preceq \nu$ with $\supp(\mu) \subseteq \mathbb{R}^{+}$. Let $(x,y^+)$, $(x,y^-)$, $(x',y') \in \Gamma$ with $-x' < -x \leq y^- < y^+$. Then
it cannot be that $y^- < y' < y^+$.
\end{corollary}

\begin{proof}
Again we draw paradigmatic graphs for the functions $f$ and $g$ from the above Lemma \ref{minlem}. It can be seen that
clearly $f(x) = f(x')$ and $g(x') > g(x)$ from which follows the corollary.

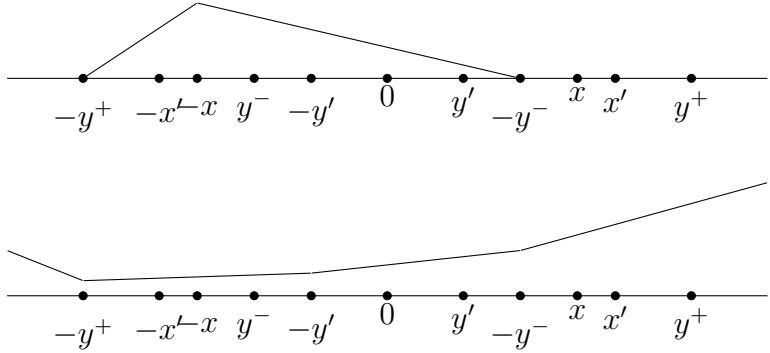
\begin{center}
\begin{tikzpicture}[every node/.style={draw,circle,fill=black,inner sep=0pt,minimum size=3pt}]
\node[label=below:$0$] at (0,0) {};
\node[label=below:$x'$] (x') at (3,0) {};
\node[label=below:$-x'$] (x') at (-3,0) {};
\node[label=below:$-y^-$] (-y-) at (1.75,0) {};
\node[label=below:$y^-$] (y-) at (-1.75,0) {};
\node[label=below:$-y^+$] (-y+) at (-4,0) {};
\node[label=below:$y^+$] (y+) at (4,0) {};
\node[label=below:$x$] (x) at (2.5,0) {};
\node[label=below:$-x$] (-x) at (-2.5,0) {};
\node[label=below:$-y'$] (-y') at (-1,0) {};
\node[label=below:$y'$] (y') at (1,0) {};
\node[draw=none,fill=none,minimum size=0pt,inner sep=0pt] (t) at (-2.5,1) {};
\draw (-y+) -- (t);
\draw (-y-) -- (t);
\draw (-5,0) -- (5,0);
\end{tikzpicture}
\end{center}

\begin{center}
\begin{tikzpicture}[every node/.style={draw,circle,fill=black,inner sep=0pt,minimum size=3pt}]
\node[label=below:$0$] at (0,0) {};
\node[label=below:$x'$] (x') at (3,0) {};
\node[label=below:$-x'$] (x') at (-3,0) {};
\node[label=below:$-y^-$] (-y-) at (1.75,0) {};
\node[label=below:$y^-$] (y-) at (-1.75,0) {};
\node[label=below:$-y^+$] (-y+) at (-4,0) {};
\node[label=below:$y^+$] (y+) at (4,0) {};
\node[label=below:$x$] (x) at (2.5,0) {};
\node[label=below:$-x$] (-x) at (-2.5,0) {};
\node[label=below:$-y'$] (-y') at (-1,0) {};
\node[label=below:$y'$] (y') at (1,0) {};
\node[draw=none,fill=none,minimum size=0pt,inner sep=0pt] (t1) at (-4,0.2) {};
\node[draw=none,fill=none,minimum size=0pt,inner sep=0pt] (t2) at (-1,0.3) {};
\node[draw=none,fill=none,minimum size=0pt,inner sep=0pt] (t3) at (1.75,0.6) {};
\node[draw=none,fill=none,minimum size=0pt,inner sep=0pt] (e1) at (-5,0.6) {};
\node[draw=none,fill=none,minimum size=0pt,inner sep=0pt] (e2) at (5,1.5) {};
\draw (e1) -- (t1);
\draw (t1) -- (t2);
\draw (t2) -- (t3);
\draw (t3) -- (e2);
\draw (-5,0) -- (5,0);
\end{tikzpicture}
\end{center}
\end{proof}

\begin{proof}[Proof of Theorem \ref{mainthm} (Minimising Part)]
Let $\underaccent{\bar}{\pi}$ be the minimising martingale transport which maximises the
secondary cost function. Let $\Gamma$ be its monotonicity set as defined in Lemma \ref{modlem}.

We first show that the mass of every point $x$ that is distributed to the right of $x$
is concentrated on the graph of an increasing function $T_u$. That means that
$\Gamma \cap \{(x,y) : x \leq y\} \subseteq \mathrm{graph}(T_u) \cup A \times \mathbb{R}$ with
$A$ countable.

Consider elements $(x,y^+),(x,y^-),(x',y') \in \Gamma$ where $y^-< x < y^+$ and $x < x' \leq y'$. Suppose
$y^- < y' < y^+$. If $y^- < -x$ we are in a situation which cannot occur according to Corollary \ref{mincor} part (i).

If we have $y^- \geq -x$ we are in a situation which Corollary \ref{mincorb} tells us is impossible. This shows that the set
on which the mass is concentrated has to be increasing.

Now suppose the mass is concentrated on more than one graph. Then we can find by Lemma \ref{comblem}
$x < x'$ and $x < y^- < y' < y^+$ with $x' < y'$ (with $(x,y^+),(x,y^-),(x',y') \in \Gamma$) which would contradict the monotonicity property of this set and therefore concludes the proof of the first statement.

Now we consider the set $\tilde{\Gamma} = \Gamma \backslash \Delta$ with $\Delta = \{(x,-x) : x \in \mathbb{R}\}$.
The properties of $\Gamma$ stated in Lemma \ref{modlem} are of course inherited by $\tilde{\Gamma}$ and thus the
corollaries above continue to hold.

Suppose we have uncountably many points $x$ such that $|\tilde{\Gamma}_x| \geq 3$. This set
will have an accumulation point (by Lemma \ref{comblem}) $x$ such that we have $y^- < y < x < y^+$ for $(x,y^-),(x,y),(x,y^+) \in \tilde{\Gamma}$. We have to consider three cases:

\begin{figure}
\begin{center}
\resizebox{9cm}{!}{
\begin{tikzpicture}[every node/.style={draw,circle,fill=black,inner sep=0pt,minimum size=3pt}]
\draw (-6,0) -- (6,0);
\draw (-6,-2) -- (6,-2);
\draw (0,0.5) -- (0,-2.5);
\node[label=above:$x$] (x) at (2,0) {};
\node[label=below:$-x$] (-x) at (-2,-2) {};
\node (y-) at (-5,-2) {};
\node[draw=none,fill=none,label=above:$y^-$] at (-5,-2.7){};
\node[label=below:$y^+$] (y+) at (4,-2) {};
\node (y) at (-3,-2) {};
\node[draw=none,fill=none,label=above:$y$] at (-3,-2.7){};
\node (y') at (-4,-2) {};
\node[draw=none,fill=none,label=above:$y'$] at (-4,-2.7){};
\node[label=above:$x'$] (x') at (1.5,0) {};
\draw (x) -- (y-);
\draw (x) -- (y);
\draw (x) -- (y+);
\draw[dashed] (x') -- (y');
\end{tikzpicture}
}
\end{center}

\begin{center}
\resizebox{9cm}{!}{
\begin{tikzpicture}[every node/.style={draw,circle,fill=black,inner sep=0pt,minimum size=3pt}]
\draw (-6,0) -- (6,0);
\draw (-6,-2) -- (6,-2);
\draw (0,0.5) -- (0,-2.5);
\node[label=above:$x$] (x) at (2.5,0) {};
\node[label=below:$-x$] (-x) at (-2.5,-2) {};
\node (y-) at (-4,-2) {};
\node[draw=none,fill=none,label=above:$y^-$] at (-4,-2.7){};
\node[label=below:$y^+$] (y+) at (4,-2) {};
\node (y) at (-1,-2) {};
\node[draw=none,fill=none,label=above:$y$] at (-1,-2.7){};
\node (y') at (-1.75,-2) {};
\node[draw=none,fill=none,label=above:$y'$] at (-1.75,-2.7){};
\node[label=above:$x'$] (x') at (3,0) {};
\draw (x) -- (y-);
\draw (x) -- (y);
\draw (x) -- (y+);
\draw[dashed] (x') -- (y');
\end{tikzpicture}
}
\end{center}

\begin{center}
\resizebox{9cm}{!}{
\begin{tikzpicture}[every node/.style={draw,circle,fill=black,inner sep=0pt,minimum size=3pt}]
\draw (-6,0) -- (6,0);
\draw (-6,-2) -- (6,-2);
\draw (0,0.5) -- (0,-2.5);
\node[label=above:$x$] (x) at (3,0) {};
\node[label=below:$-x$] (-x) at (-3,-2) {};
\node[label=below:$-x'$] (-x) at (-2.25,-2) {};
\node (y-) at (-1.5,-2) {};
\node[draw=none,fill=none,label=above:$y^-$] at (-1.5,-2.7){};
\node[label=below:$y^+$] (y+) at (4,-2) {};
\node (y) at (-1,-2) {};
\node[draw=none,fill=none,label=above:$y$] at (-1,-2.7){};
\node (y') at (-0.5,-2) {};
\node[draw=none,fill=none,label=above:$y'$] at (-0.5,-2.7){};
\node[label=above:$x'$] (x') at (2.25,0) {};
\draw (x) -- (y-);
\draw (x) -- (y);
\draw (x) -- (y+);
\draw[dashed] (x') -- (y');
\end{tikzpicture}
}
\end{center}

\caption{The three cases that would contradict the proposed structure for the left part of the transport plan}
\end{figure}
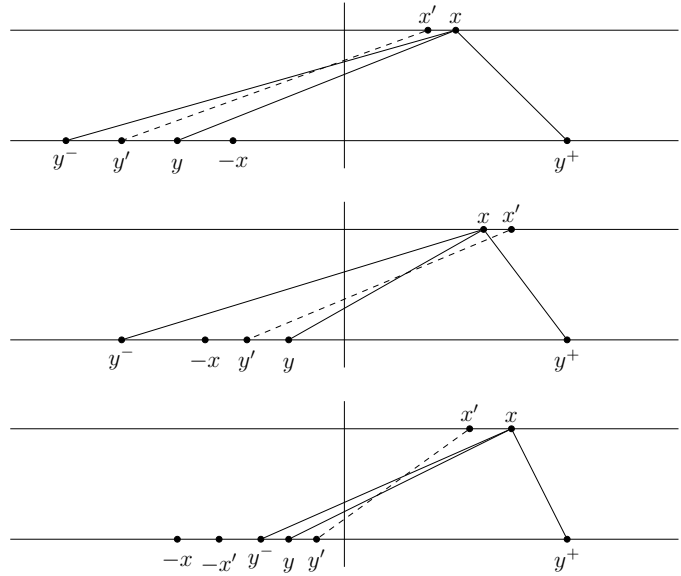

\begin{itemize}
\item $y^- < y < -x$: In this case we can find $x'<x$ and $y^- < y' < -x$ with $(x',y') \in \tilde{\Gamma}$ by Lemma \ref{comblem}. By Corollary \ref{mincor} (ii) this is impossible.
\item $y^- < -x < y$: Another application of Lemma \ref{comblem} gives $x < x'$ and $-x < y' < y^+$ with $(x',y') \in \tilde{\Gamma}$.
This is a contradiction to Corollary \ref{mincor} (i).
\item $-x < y^- < y$: Again by Lemma \ref{comblem} we find $x' < x$ and $y'$ with $(x',y') \in \tilde{\Gamma}$ such that $-x' < y^-$ and $y^- < y' < y^+$. The impossibility of this scheme can be seen by an application of Corollary \ref{mincorb}.
\end{itemize}

This shows that the the mass that is distributed to the left and is not mapped to the secondary diagonal 
is concentrated on a graph.
\end{proof}

Knowledge about the structure can be used to formulate differential equations to calculate
$T_u$ and $T_l$. This is only possible if the density of the marginal measures is analytically
tractable. Therefore there will only be an explicit solution in some special cases. For illustrative
purposes we give the graphs of $T_u$ and $T_d$ for very simple marginal measures in Figure \ref{supfig}.

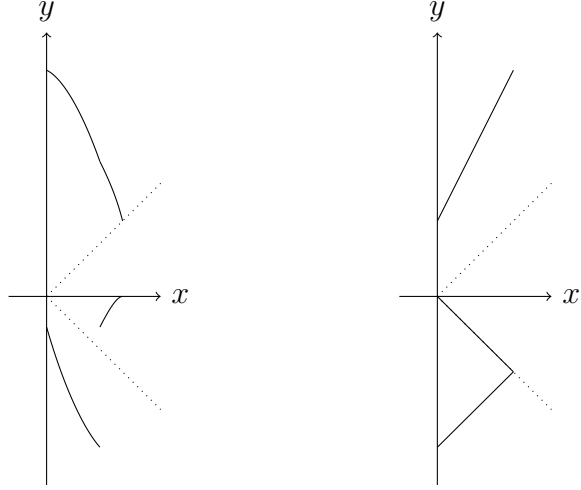
\begin{figure}
\begin{center}
\begin{tikzpicture}
\draw[->] (-0.5,0) -- (1.5,0) node[right] {$x$};
\draw[->] (0,-2.5) -- (0,3.5) node[above] {$y$};
\draw[dotted] (0,0) -- (1.5,1.5);
\draw[dotted] (0,0) -- (1.5,-1.5);
\draw[domain=0:0.7014008587679633,variable=\x] plot ({\x},{1.29684-2*\x+0.5*sqrt(11.603+20.7495*\x-24*\x*\x)});
\draw[domain=0:0.7014008587679633,variable=\x] plot ({\x},{1.29684-2*\x-0.5*sqrt(11.603+20.7495*\x-24*\x*\x)});

\draw[domain=0.7014008587679633:1,variable=\x] plot ({\x},{(15-20*\x+6*\x*\x)/(5-4*\x)});
\draw[domain=0.7014008587679633:1,variable=\x] plot ({\x},{(10*\x*\x-20*\x+10)/(4*\x-5)});
\end{tikzpicture} \qquad \qquad \qquad
\begin{tikzpicture}
\draw[->] (-0.5,0) -- (1.5,0) node[right] {$x$};
\draw[->] (0,-2.5) -- (0,3.5) node[above] {$y$};
\draw[dotted] (0,0) -- (1.5,1.5);
\draw[dotted] (0,0) -- (1.5,-1.5);
\draw (0,-2) -- (1,-1);
\draw (0,0) -- (1,-1);
\draw (0,1) -- (1,3);
\end{tikzpicture}
\end{center}
\caption{Support of the $|x+y|$ maximising and minimising martingale transports for $\mu=\mathds{1}_{[0,1]}\#\lambda$
and $\nu=\frac{1}{4}\mathds{1}_{[-2,0]\cup[1,3]}\#\lambda$ in the $(x,y)$-plane.}
\label{supfig}
\end{figure}

\subsection{Maximising Transport Layout}
\label{secmax}

The proof of the structure of the maximising transport layout is very similar to the
minimiser. We again need a secondary optimisation problem to get rid of non-uniqueness
of the solution. This time we look among all maximisers for the cost function $c(x,y):=|x+y|$
for the one minimising the value of $\int c(x,y)y^2 \,d\pi(x,y)$. We again have a lemma which
we will use in two corollaries to obtain general constellations which are impossible in
the optimal transport plan. We will omit the proofs of the corollaries because they work exactly the
same way as in the minimising case.

Note that a maximiser for $|x+y|$ is a minimiser for $-|x+y|$ so we can continue to use
Lemma \ref{modlem} in the natural way, even though we now have a maximisation problem with
a secondary minimisation problem instead of the other way around.

\begin{lemma}
\label{maxlem}
Let $\pi$ be a maximising martingale transport plan which is minimising the
secondary optimisation problem for given marginal distributions $\mu$ and $\nu$ in
convex order and let $\Gamma$ be its monotonicity set as given by Lemma
\ref{modlem}. Let $y^- < y < y^+ \in \supp(\nu)$ and some $\lambda \in [0,1]$
such that
$y = (1-\lambda)y^- + \lambda y^+$ holds. Consider the functions
\begin{align*}
f(t) &= (1-\lambda)|t+y^-|+\lambda|t+y^+| - |t+y| \qquad \text{and} \\
g(t) &= (1-\lambda)|t+y^-|(y^-)^2 + \lambda|t+y^+|(y^+)^2 - |t+y|y^2 \text{.}
\end{align*}

\begin{enumerate}[(i)]
\item There do not exist $x,x'$ such that $(x,y^-),(x,y^+),(x',y) \in \Gamma$ and $f(x) < f(x')$.

\item There do not exist $x,x'$ such that $(x,y^-),(x,y^+),(x',y) \in \Gamma$ and $f(x) = f(x')$ with
$g(x) > g(x')$.
\end{enumerate}
\end{lemma}

\begin{proof}
Obviously it follows from $f(x) < f(x')$ that we have
\[(1-\lambda)|x+y^-|+\lambda|x+y^+| + |x'+y| < (1-\lambda)|x'+y^-|+\lambda|x'+y^+| + |x+y| \text{.}\]
So if we again construct $\alpha$ and $\alpha'$ as in the proof to Lemma \ref{minlem}, then in this case we have
$\int c \,d\alpha < \int c \,d\alpha'$ and therefore a contradiction to Lemma \ref{modlem}

Again the same argument holds for the second part using $f$ and $g$.
\end{proof}

\begin{corollary}
\label{maxcor}
Consider some maximising martingale transport $\pi$ which minimises the secondary
optimisation problem for marginals $\mu \preceq \nu$ with $\supp(\mu) \subseteq \mathbb{R}^{+}$ and
monotonicity set $\Gamma$. Let
$(x,y^+)$,$(x,y^-)$,$(x',y')$ $\in \Gamma$ such that $y^- < -x' < y^+$ and $y^- < y' < y^+$. Then it cannot be
that 
\begin{enumerate}[(i)]
\item $y' \leq -x'$ and $x < x'$ or
\item $-x' \leq y'$ and $x' < x$.
\end{enumerate}
\end{corollary}

\begin{proof}
Analogous to the case of a minimising martingale transport.
\end{proof}

\begin{corollary}
\label{maxcorb}
Consider some maximising (and secondary minimising) martingale transport 
$\pi$ for marginals $\mu \preceq \nu$ with $\supp(\mu) \subseteq \mathbb{R}^{+}$. 
Let $\Gamma$ be its monotonicity set. 
Let $(x,y^+),(x,y^-),(x',y') \in \Gamma$ with $x' < x$ and $-x \leq y^-$. Then
it cannot be that $y^- < y' < y^+$.
\end{corollary}

\begin{proof}
Again the proof is exactly analogous to the case of the minimising transport.
\end{proof}

\begin{proof}[Proof of Theorem \ref{mainthm} (Maximising Part)]
Similar to the minimiser, we first show that the mass of every point $x$ that is distributed to the right of $x$
is concentrated on the graph of a decreasing function $T_u$ (in the same sense given in the proof of the minimising part).

Consider elements $(x,y^+),(x,y^-),(x',y') \in \Gamma$ where $y^-< x \leq y^+$ and $x' < x$ with $x' \leq y'$. Suppose
$y^- < y' < y^+$. If $y^- < -x$ we are in a situation which cannot occur according to Corollary \ref{maxcor} part (ii).

If we have $y^- \geq -x$ we are in a situation which Corollary \ref{maxcorb} tells us is impossible. This shows that the set
$\Gamma \cap \{(x,y) : y > x\}$ is
monotone decreasing. I.e.\ for $x < x'$ and $y \in \Gamma_x \cap \{z : z > x\}$, $y' \in \Gamma_{x'} \cap \{z : z > x'\}$
we have $y' \leq y$. That it is concentrated on a single graph then follows in the same fashion as for the minimiser.

Suppose now, we have uncountably many points that are distributed by $\pi$ to more than two points
(i.e.\ $|\Gamma_x| \geq 3$). This set
will then have an accumulation point $x$ such that we have $y^- < y < x < y^+$ for $(x,y^-),(x,y),(x,y^+) \in \Gamma$. We have to consider the same three cases as for the minimiser with slightly different arguments and an additional case which occurs
because we cannot exclude the secondary diagonal as we did with the minimiser:

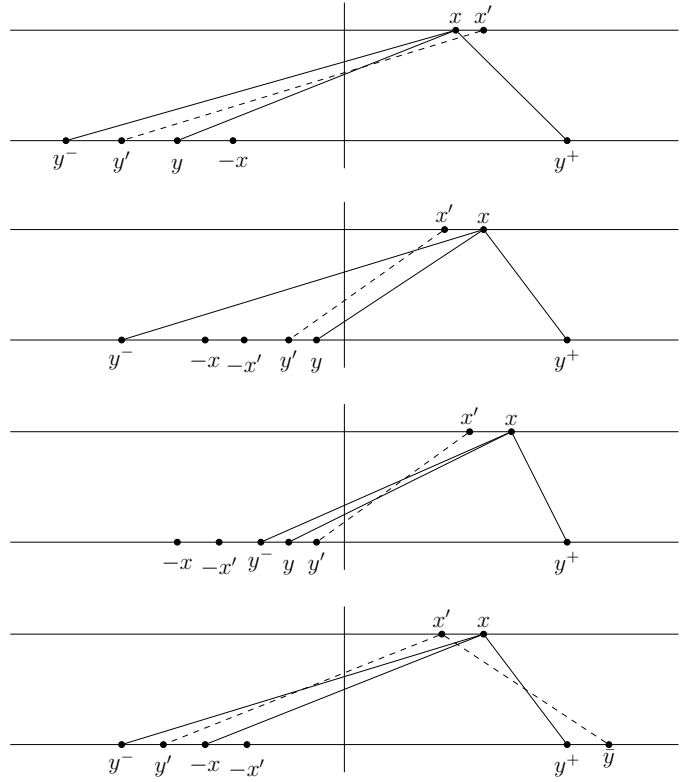
\begin{figure}
\begin{center}
\resizebox{9cm}{!}{
\begin{tikzpicture}[every node/.style={draw,circle,fill=black,inner sep=0pt,minimum size=3pt}]
\draw (-6,0) -- (6,0);
\draw (-6,-2) -- (6,-2);
\draw (0,0.5) -- (0,-2.5);
\node[label=above:$x$] (x) at (2,0) {};
\node[label=below:$-x$] (-x) at (-2,-2) {};
\node (y-) at (-5,-2) {};
\node[draw=none,fill=none,label=above:$y^-$] at (-5,-2.7){};
\node[label=below:$y^+$] (y+) at (4,-2) {};
\node (y) at (-3,-2) {};
\node[draw=none,fill=none,label=above:$y$] at (-3,-2.7){};
\node (y') at (-4,-2) {};
\node[draw=none,fill=none,label=above:$y'$] at (-4,-2.7){};
\node[label=above:$x'$] (x') at (2.5,0) {};
\draw (x) -- (y-);
\draw (x) -- (y);
\draw (x) -- (y+);
\draw[dashed] (x') -- (y');
\end{tikzpicture}
}
\end{center}

\begin{center}
\resizebox{9cm}{!}{
\begin{tikzpicture}[every node/.style={draw,circle,fill=black,inner sep=0pt,minimum size=3pt}]
\draw (-6,0) -- (6,0);
\draw (-6,-2) -- (6,-2);
\draw (0,0.5) -- (0,-2.5);
\node[label=above:$x$] (x) at (2.5,0) {};
\node[label=below:$-x$] (-x) at (-2.5,-2) {};
\node (y-) at (-4,-2) {};
\node[draw=none,fill=none,label=above:$y^-$] at (-4,-2.7){};
\node[label=below:$y^+$] (y+) at (4,-2) {};
\node (y) at (-0.5,-2) {};
\node[draw=none,fill=none,label=above:$y$] at (-0.5,-2.7){};
\node (y') at (-1,-2) {};
\node[draw=none,fill=none,label=above:$y'$] at (-1,-2.7){};
\node[label=above:$x'$] (x') at (1.8,0) {};
\node[label=below:$-x'$] (-x') at (-1.8,-2) {};
\draw (x) -- (y-);
\draw (x) -- (y);
\draw (x) -- (y+);
\draw[dashed] (x') -- (y');
\end{tikzpicture}
}
\end{center}

\begin{center}
\resizebox{9cm}{!}{
\begin{tikzpicture}[every node/.style={draw,circle,fill=black,inner sep=0pt,minimum size=3pt}]
\draw (-6,0) -- (6,0);
\draw (-6,-2) -- (6,-2);
\draw (0,0.5) -- (0,-2.5);
\node[label=above:$x$] (x) at (3,0) {};
\node[label=below:$-x$] (-x) at (-3,-2) {};
\node[label=below:$-x'$] (-x) at (-2.25,-2) {};
\node (y-) at (-1.5,-2) {};
\node[draw=none,fill=none,label=above:$y^-$] at (-1.5,-2.7){};
\node[label=below:$y^+$] (y+) at (4,-2) {};
\node (y) at (-1,-2) {};
\node[draw=none,fill=none,label=above:$y$] at (-1,-2.7){};
\node (y') at (-0.5,-2) {};
\node[draw=none,fill=none,label=above:$y'$] at (-0.5,-2.7){};
\node[label=above:$x'$] (x') at (2.25,0) {};
\draw (x) -- (y-);
\draw (x) -- (y);
\draw (x) -- (y+);
\draw[dashed] (x') -- (y');
\end{tikzpicture}
}
\end{center}

\begin{center}
\resizebox{9cm}{!}{
\begin{tikzpicture}[every node/.style={draw,circle,fill=black,inner sep=0pt,minimum size=3pt}]
\draw (-6,0) -- (6,0);
\draw (-6,-2) -- (6,-2);
\draw (0,0.5) -- (0,-2.5);
\node[label=above:$x$] (x) at (2.5,0) {};
\node[label=below:$-x$] (-x) at (-2.5,-2) {};
\node[label=below:$-x'$] (-x') at (-1.75,-2) {};
\node (y-) at (-4,-2) {};
\node[draw=none,fill=none,label=above:$y^-$] at (-4,-2.7){};
\node[label=below:$y^+$] (y+) at (4,-2) {};
\node[label=below:$\bar{y}$] (-y) at (4.75,-2) {};
\node (y') at (-3.25,-2) {};
\node[draw=none,fill=none,label=above:$y'$] at (-3.25,-2.7){};
\node[label=above:$x'$] (x') at (1.75,0) {};
\draw (x) -- (y-);
\draw (x) -- (-x);
\draw (x) -- (y+);
\draw[dashed] (x') -- (y');
\draw[dashed] (x') -- (-y);
\end{tikzpicture}
}
\end{center}

\caption{The four cases that would contradict the proposed structure for the left part of the transport plan}
\end{figure}

\begin{itemize}
\item $y^- < y < -x$: In this case we can find $x<x'$ and $y^- < y' < -x$ with $(x',y') \in \Gamma$ by Lemma \ref{comblem}. By Corollary \ref{maxcor} (i) this is impossible.
\item $y^- < -x < y$: Another application of Lemma \ref{comblem} gives us $x' < x$ and $-x < -x' < y' < y^+$ with $(x',y') \in \Gamma$. This is a contradiction to Corollary \ref{maxcor} (ii).
\item $-x \leq y^- < y$: Again by Lemma \ref{comblem} we find $x' < x$ and $y'$ with $(x',y') \in \Gamma$ such that $-x \leq y^- < y' < y^+$. For $-x' \leq y^-$ we can see the impossibility of this scheme by an application of Corollary \ref{maxcorb};
for $y^- < -x'$ we obtain the desired contradiction by Corollary \ref{maxcor} (ii).
\item $y^- < y = -x$: Here we find $(x',\bar{y}), (x',y') \in \Gamma$ with $x' < x$, $x' < x < y^+ < \bar{y}$ and also
$y^- < y'$. Now, if $y' < -x'$ holds, the elements $(x',\bar{y}),(x',y')$ and $(x,y^+)$ are a contradiction to
Corollary \ref{maxcor} (i). If, instead $y' \geq -x'$ holds, the same elements are a contradiction to Corollary \ref{maxcorb}.
\end{itemize}

So we have indeed shown that the mass is concentrated on two graphs.
\end{proof}

\begin{remark}
It is also rather straightforward to show that the graph of $T_l$ is decreasing if it is below the secondary diagonal
and increasing if it is above the secondary diagonal.
\end{remark}

\section{Notes on the Continuous Time Case}

\subsection{The one marginal case}

We will now discuss the case of an Asian option on a forward price of an asset
with continuous averaging. 
Given a process $(X_t)_{0 \leq t \leq T}$ (with continuous paths) on some time
interval $[0,T]$, we consider the option with payoff
\[\Phi((X_t)_{0 \leq t \leq T}) := \phi\left(\frac{1}{T}\int_0^TX_t dt\right) \]
where $\phi: \mathbb{R} \to \mathbb{R}$ is a convex lower semi-continuous function. We again
try to find optimal upper and lower bounds for the possible prices of $\Phi$. We do this by providing
candidates to the primal and the dual problem as outlined in Section 2 and checking that they agree.
So we provide a price given by a consistent martingale measure as well as a super-hedge for the same price.
The maximising and minimising models will not have continuous paths though, but finitely many jumps
at deterministic points in time 
at which the paths are either left- or right continuous. We will therefore first prove a lemma to show
that we can approximate the value obtained by such a model by the values from continuous models.
For simplicity we will assume a model with
a single jump only. More concretely we will show the following lemma.

\begin{lemma}
\label{approxlem}
Consider two random variables $(X,Y)$ and a probability law $\mathbb{Q}$ such that $(X,Y)$ is
a two step martingale with respect to $\mathbb{Q}$. 
Let $\phi : \mathbb{R} \to \mathbb{R}$ be a lower semi-continuous convex function. Then we can find a sequence $(Y^n_t)_{t\in [0,T]}$ of continuous time
martingales with continuous paths on $[0,T]$ such that $\lim_{n\to \infty} E\left[\phi\left(\frac{1}{T}\int_0^TY^n_t dt\right)\right]
= E\left[\phi\left(\frac{t_1}{T}X + \frac{T-t_1}{T}Y\right)\right].$
\end{lemma}

If we have a process $(X_t)$ with a jump at $t_1$ in the form that we set $X_t = X$ for $0\leq t < t_1$
and $X_t = Y$ for $t_1 \leq t \leq T$, then the value achieved by it is clearly the limiting value of the sequence $Y_n$.
We can obtain a similar result for a left continuous jump as well as for finitely many jumps. 

For the proof we will need another lemma which we assume is well known, but as we could not find it in the
literature we will still provide a proof here.

\begin{lemma}
\label{convlem}
Let $\nu,\mu,\mu_1,\mu_2, \dots$ be probability measures such that
$\mu,$ $\mu_1,$ $\mu_2,\dots \preceq \nu$ and $\mu_n$ converges to $\mu$ weakly.
Suppose that $\nu$ is integrable and for some convex function $\phi : \mathbb{R} \to \mathbb{R}$
we have $\int \phi \,d\nu < \infty$. Then 
$\lim_{n \to \infty} \int \phi \, d\mu_n = \int \phi \, d\mu$ holds.
\end{lemma}

\begin{proof}
First we will assume that $\phi \geq 0$.
Now define the auxiliary function $\phi_m,\psi_m : \mathbb{R} \to \mathbb{R}$ by
$\phi_m := (\phi -m)_+$ and $\psi_m := \phi - \phi_m$.
Observe that $\phi_m$ is clearly convex for all $m$. Furthermore we have that $\psi_m$ is bounded
by $m$ and also $\psi_m$ is a monotone increasing sequence which converges
to $\phi$ pointwise. Therefore, by convex order and monotone convergence we can choose $m$ large enough such that for 
fixed $\varepsilon > 0$ we have
\[(0 \leq) \int \phi_m \, d\mu \leq \int \phi_m \, d\nu = \int \phi - \psi_m \, d\nu < \varepsilon \text{.}\]
Similarly we have $\int \phi_m \, d\mu_n < \varepsilon$.
Now, because $\psi_m$ is bounded, we have by weak convergence that
\[\lim_{n\to \infty} \int \psi_m \, d\mu_n(x) = \int \psi_m \, d\mu \text{.}\]
This means that we can choose $n$ large enough such that $|\int \psi_m \, d(\mu-\mu_n)| < \varepsilon$
Putting this together we can conclude that
\[\left|\int \phi \,d(\mu-\mu_n)\right| \leq \int \phi_m \,d\mu + \int \phi_m \,d\mu_n  + \left|\int \psi_m \, d(\mu-\mu_n)\right|
< 3 \varepsilon \text{.}\]
For general convex $\phi$ we can find an affine function $g(x) := ax +b$ such that
$\psi := \phi + g \geq 0$. As $\psi$ is clearly convex and satisfies $\int \phi \,d\nu < \infty$,
we have $\lim_{n \to \infty} \int \psi \, d\mu_n = \int \psi \, d\mu$. Furthermore we
have $\int g \, d\mu = \int g \, d\mu_n$ for all $n$ because $\mu_n$ and $\mu$ have the same mean
as $\nu$ by the definition of convex order. From this we can conclude that also 
$\lim_{n \to \infty} \int \phi \, d\mu_n = \int \phi \, d\mu$ holds.
\end{proof}

\begin{proof}[Proof of Lemma \ref{approxlem}]
Define the process $(Y^n_t)_{t\in[0,T]}$ by setting $Y^n_t = X$ for $0 \leq t \leq t_1-1/n$ and $Y^n_t = Y$
for $t_1 \leq t \leq T$. For $t_1-1/n < t < t_1$ we can interpolate between $X$ and $Y$ in such a way such that the process
is a continuous martingale. There are various constructions for this, see for instance \cite{He77}. Now
for this process we can easily calculate
\begin{align*}
&E\left[\left|\int_0^TY^n_t\,dt - \frac{t_1}{T}X - \frac{T-t_1}{T}Y\right|\right] \\
& \qquad= \frac{1}{T} E\left[\int_{t_1-1/n}^{t_1} X_t-X\,dt\right]
\leq \frac{1}{T}\int_{t_1-1/n}^{t_1} E[|X_t|]+E[|X|] \, dt \\
&\qquad \leq \frac{1}{nT}E[|Y|] \overset{n\to \infty}{\rightarrow} 0 \text{.}
\end{align*}
So we have convergence in $L^1$ which implies weak convergence of the laws of $\int_0^TY^n_t\,dt$ to
the law of $\frac{t_1}{T}X + \frac{T-t_1}{T}Y$. We also have $\int_0^TY^n_t\,dt, \frac{t_1}{T}X + \frac{T-t_1}{T}Y
\preceq Y$ and $E[\phi(Y)] < \infty$. By Lemma \ref{convlem} we obtain the desired result.
\end{proof}

In a first step we assume that the information we can use for determining the price for this option are vanilla calls
with arbitrary strikes and maturity $T$. By Breeden and Litzenberger \cite{BrLi78} this means that we know the distribution
of $X_T$. We furthermore assume that the law of $X_0$ is given by $\delta_{E[X_T]}$.
With this observation we can easily determine an upper and a lower bound on the possible prices
for $\Phi$. For any possible law we see that
\[E\left[\phi\left(\int_0^TX_t\frac{dt}{T}\right)\right] \leq E\left[\int_0^T\phi(X_t)\frac{dt}{T}\right] \leq \int_0^TE[\phi(X_T)]\frac{dt}{T} = E[\phi(X_T)],\]
\[E\left[\phi\left(\int_0^TX_t\frac{dt}{T}\right)\right] \geq \phi\left(\int_0^TE[X_t]\frac{dt}{T}\right) = \phi(X_0).\]
These bounds are assumed by a joint law for the process such that $X_t = X_T$ for $0 < t \leq T$ and $X_t = X_0$ for
$0\leq t < T$ respectively. This means that we either jump at the beginning to the law of $X_T$ and then stay constant
or that we stay constant until the very end and then jump to the distribution of the terminal value. This means that no
robust superhedge can cost less than $E[\phi(X_T)]$.

The approximation via Lemma \ref{approxlem} then yields that this is indeed the robust bound for processes
with continuous paths.

In this case we can also devise optimal pathwise sub- and superhedging strategies. The subhedge is essentially trivial
by just not investing in the asset. The superhedge can be derived by the following
calculation. For simplicity we assume $\phi$ to be
differentiable. Note that because of the convexity of $\phi$ this is the case almost anywhere and on the remaining
points the derivative we use here can be substituted with an arbitrary element from the subdifferential at this point
\begin{align*}
\phi\left(\int_0^T X_t \frac{dt}{T}\right) &\leq \int_0^T \phi(X_t) \frac{dt}{T} =\phi(X_T) - \int_0^T(\phi(X_T) - \phi(X_t))\frac{dt}{T} \\
&\leq \phi(X_T) - \int_0^T\phi'(X_t)(X_T-X_t)\frac{dt}{T} \\
&= \phi(X_T) - \int_0^T\int_t^T\phi'(X_t)dX_s\frac{dt}{T} \\
& = \phi(X_T) - \int_0^T\left(\frac{1}{T}\int_0^t\phi'(X_s)ds\right)dX_t \text{.}
\end{align*}
The price of this super-hedge is obviously $E[\phi(X_T)] (=\mathcal{P}(\phi(X_T)))$ which makes it optimal. The first term can be replicated
with a portfolio of vanilla calls and puts. The second term gives a trading strategy in the asset, therefore this is really
a super-hedge with the least possible costs.

Note again that the above calculation should be understood to be a pathwise derivation. This means that the
integrals with respect to $dX_t$ are not meant to be Itô-Integrals but regular Lebesgue-Stieltjes-Integrals
with respect to a path of $X_t$. These are defined because the integrands are of bounded total variation
($\phi'(X_t)$ is constant and $\int_0^t \phi'(X_s) ds$ is absolutely continuous). The change of the order of integration
can then be derived by writing the integral explicitly as a limit operation.

\subsection{The two marginal case}

We proceed with considering the next possible step for this problem and assume that we are also given
the market prices for vanilla calls at some intermediate time $0 < t_1 < t_2 = T$. This problem does not appear
to give simple answers to the questions discussed above. Therefore we will not discuss the more intricate
n-marginal case.

From the results in the case with one prescribed marginal law, one might assume that the maximum value is achieved by
an initial jump to $X_{t_1}$ at the beginning and a jump to $X_{t_2}$ right after $t_1$ which is done in a certain
optimal way such that the process remains a martingale. The minimum value would then be achieved
by instead staying constant for the time interval $[0,t_1)$, then jumping to $X_{t_1}$ and staying
constant again in $[t_1,t_2)$ before jumping to $X_{t_2}$ (where the last jump does not contribute anything
to the final value of the integral). This however turns out to be wrong in the
case of the minimiser and we can find easy
counterexamples where it is possible to achieve smaller values then the smallest value which can be attained by
a strategy following the above restrictions.

\begin{example}
We consider the problem with $\phi=|\cdot|$ and marginals
\[\mu_1 = \mu_2 := (\delta_{-2}+\delta_{-1}+\delta_{1}+\delta_{2})/4 \text{.}\]
The process then has to start in $\int x \, d\mu_1(x) = 0$. The price according to the process which we would hope
to be the minimiser would then be $\int |x| \, d\mu_1(x) = 3/2$.
Alternatively we define a two dimensional random variable $(Y,Z)$ which has law
\[\pi := \frac{1}{4}(\delta_{(1/4,-1)}+\delta_{(-1/4,1)}) + 
	\frac{5}{28}(\delta_{(1/4,2)}+\delta_{(-1/4,-2)}) + \frac{1}{14}(\delta_{(0,2)}+\delta_{(0,-2)}) \text{.}\]
Now we define the process by $X_0 := 0$, $X_t := Y$ for $0 < t < 1$ and $X_t := Z$ for $1 \leq t \leq 2$.
This process does have the right marginals and is a martingale but implies the price $\frac{41}{28} < \frac{3}{2}$. 
Therefore the above described scheme is not optimal in general.
\end{example}

The only thing we can indeed easily show is that one minimising model will have the property that
$X_t = X_{t_1}$ for $t_1 \leq t < t_2$ while one maximising model fulfills $X_t = X_{t_2}$ for $t_1 < t \leq t_2$.
This follows in the case of the minimising models by an application of Jensen's inequality for conditional 
expectation and the martingale property of $X_t$:
\begin{align*}
E\left[\phi\left(\int_0^{T}X_t \frac{dt}{T}\right)\right] &= E\left[E\left[\phi\left(\left.\int_0^{T}X_t \frac{dt}{T}\right)\right|\mathbb{F}_{t_1}\right]\right] \\
&\geq E\left[\phi\left(\int_0^{T}E[X_t|\mathbb{F}_{t_1}] \frac{dt}{T}\right)\right] \\
&= E\left[\phi\left(\frac{1}{T}\left(\int_0^{t_1}X_t dt + X_{t_1}(T-t_1)\right)\right)\right] \text{.}
\end{align*}
The calculation in the case of the maximiser is somewhat more intricate, but equally elementary

\begin{align*}
&E\left[\phi\left(\int_0^{T}X_t \frac{dt}{T}\right)\right] \\
	&\quad = E\left[\phi\left(\int_{t_1}^{T}\frac{1}{T}\left(\int_0^{t_1}X_s\frac{ds}{T-t_1} + X_t\right)dt\right)\right] \\
	&\quad \leq \int_{t_1}^T E\left[\phi\left(\frac{1}{T}
		\left(\int_0^{t_1}X_sds + (T-t_1)X_t\right)\right)\right]\frac{dt}{T-t_1} \\
	&\quad\leq \int_{t_1}^T E\left[\phi\left(\frac{1}{T}\left(\int_0^{t_1}X_sds + (T-t_1)X_{T}\right)\right)\right] \\
	& \qquad -E\left[\phi'\left(\frac{1}{T}\left(\int_0^{t_1}X_sds + (T-t_1)X_t\right)\right) \right.\\
	&\qquad \qquad \qquad\left. \times \vphantom{\int_0^{t_1}}\frac{T-t_1}{T}E[X_T - X_t|\mathbb{F}_t]\right] \frac{dt}{T-t_1} \\
	&\quad= E\left[\phi\left(\frac{1}{T}\left(\int_0^{t_1}X_sds + (T-t_1)X_{T}\right)\right)\right]\text{.}
\end{align*}
Now, for the maximiser, we were not able to find a counterexample to the conjecture above. We
will therefore keep it as a conjecture. More concretely formulated the conjecture is:

\begin{conjecture}
Suppose $(X_t)_{t \in [0,T]}$ is a martingale on some measurable filtered space with respect to some
law $\mathbb{Q}$. For fixed $0 < t_1 < T$ set $\mu := \mathrm{Law}(X_{t_1})$ and $\nu := \mathrm{Law}(X_T)$.
Choose $\pi \in \mathcal{M}(\mu,\nu)$ such that 
\[\int \phi(t_1x+(T-t_1)y) \, d(\pi-\pi')(x,y) \geq 0 \quad \forall \pi' \in \mathcal{M}(\mu,\nu) \text{.}\]
Then it holds for any $0 \leq t \leq t_1$ that
\[E_{\mathbb{Q}}[\phi(t_1X_t + (T-t_1)X_T)] \leq \int \phi(t_1x+(T-t_1)y) \, d\pi(x,y) \text{.} \]
\end{conjecture}

If this is true, we can proceed in the above calculation in the following way:
\begin{align*}
&E\left[\phi\left(\frac{1}{T}\left(\int_0^{t_1}X_sds + (T-t_1)X_{T}\right)\right)\right] \\
&\quad= E\left[\phi\left(\int_0^{t_1} \frac{1}{T}\left(t_1 X_s + (T-t_1)X_{T}\right)\frac{ds}{t_1}\right)\right] \\
&\quad \leq \int_0^{t_1} E\left[ \phi\left(\frac{1}{T}\left(t_1 X_s + (T-t_1)X_{T}\right) \right) \right] \frac{ds}{t_1} \\
&\quad \overset{!}{\leq} E_{\pi}\left[\phi\left(\frac{1}{T}\left(t_1 X_{t_1} + (T-t_1)X_{T}\right) \right) \right] \text{,}
\end{align*}
where we used the conjecture only for the last inequality and chose $\pi$ as the maximising martingale transport
plan.

While we cannot prove this, so far, we will prove an important special case of this conjecture which we will formulate
in the following
\begin{theorem}
\label{domthm}
Suppose $(X_t)_{t \in [0,2]}$ is a martingale on some measurable filtered space with respect to some
law $\mathbb{Q}$. Set $\mu := \mathrm{Law}(X_1)$ and $\nu := \mathrm{Law}(X_2)$.
Choose $\pi \in \mathcal{M}(\mu,\nu)$ such that 
\[\int (x+y-K)_+ \, d(\pi-\pi')(x,y) \geq 0 \quad \forall \pi' \in \mathcal{M}(\mu,\nu) \text{.}\]
Then it holds for any $0 \leq t \leq 1$ that
\[E_{\mathbb{Q}}[(X_t+ X_2-K)_+] \leq \int (x+y-K)_+ \, d\pi(x,y) \text{.} \]
\end{theorem}

We will prove this by manually computing the result for finite measures of a special form and disintegrating
the original measures into parts of this form. We will call these finite (martingale) measures
\begin{definition}[Binomial Transport Plan]
Let $x^- \leq x \leq x^+$ be arbitrary real numbers. Additionally, $y^{--} \leq x^- \leq y^{-+}$ and
$y^{+-} \leq x^+ \leq y^{++}$ are also arbitrary. If $x^+ > x^-$, $y^{-+} > y^{--}$ and 
$y^{++} > y^{+-}$, define the factors
$\lambda^+ := \frac{x - x^-}{x^+ - x^-} = 1 - \lambda^-$, 
$\lambda^{-+} = \frac{x^- - y^{--}}{y^{-+}-y^{--}} = 1-\lambda^{--}$ and
$\lambda^{++} = \frac{x^+ - y^{+-}}{y^{++}-y^{+-}} = 1-\lambda^{+-}$.
 Otherwise just set
$\lambda^+ = 1$, $\lambda^{-+}=1$ or $\lambda^{++}=1$, respectively.
Then we will call the three step martingale transport plan
\begin{align*}
\pi & := \lambda^{-}\lambda^{--}\delta_{(x,x^-,y^{--})} +\lambda^{-}\lambda^{-+}\delta_{(x,x^-,y^{-+})} \\
& \qquad + \lambda^{+}\lambda^{+-}\delta_{(x,x^+,y^{+-})} +\lambda^{+}\lambda^{++}\delta_{(x,x^+,y^{++})} 
\end{align*}
a \emph{Binomial Transport Plan} and abbreviate it with \emph{BTP}. We call the \emph{right part} of $\pi$ the two step martingale transport plan
\begin{align*}
\pi_r & := \lambda^{-}\lambda^{--}\delta_{(x^-,y^{--})} +\lambda^{-}\lambda^{-+}\delta_{(x^-,y^{-+})} \\
& \qquad + \lambda^{+}\lambda^{+-}\delta_{(x^+,y^{+-})} +\lambda^{+}\lambda^{++}\delta_{(x^+,y^{++})} \text{.}
\end{align*}
The \emph{left part} of $\pi$ will be the two step martingale transport plan
\begin{align*}
\pi_l & := \lambda^{-}\lambda^{--}\delta_{(x,y^{--})} +\lambda^{-}\lambda^{-+}\delta_{(x,y^{-+})} \\
& \qquad + \lambda^{+}\lambda^{+-}\delta_{(x,y^{+-})} +\lambda^{+}\lambda^{++}\delta_{(x,y^{++})} \text{.}
\end{align*}
\end{definition}

Now we will reformulate Theorem \ref{domthm} for the BTP's. The following lemma and
the proof are formulated for $|x+y|$ instead of $(x+y - K)_+$ which is equivalent because it holds for arbitrary BTP's.
See Section \ref{secdis} for an explanation of this fact.

\begin{lemma}[Dominance of the Right Part of a BTP]
\label{domlem}
Let $\pi$ be an arbitrary BTP with nodes $(x,x^-,x^+,y^{--},y^{-+},y^{+-},y^{++})$. Then
one of the following has to hold:
\begin{enumerate}[(i)]
\item $\pi_r$ is a suboptimal martingale transport plan for the 
cost function $c(x,y):=|x+y|$. I.e.\ there exists a competitor
$\tilde{\pi}_r$ such that $\int |x+y|\, d\pi_r(x,y) < \int |x+y|\,d\tilde{\pi}_r(x,y)$.
\item $\pi_r$ has higher cost for the cost function $|x+y|$ than $\pi_l$. I.e.\
$\int |x+y|\, d\pi_l(x,y) \leq \int |x+y|\,d\pi_r(x,y)$.
\end{enumerate}
\end{lemma}

We will defer the proof to later. Using Lemma \ref{domlem} we can now easily prove 
Theorem \ref{domthm}.

\begin{proof}[Proof of Theorem \ref{domthm}]
Note that $(X_t,X_1,X_2)$ is a three-step martingale with respect to $\mathbb{Q}$. We can forget about
the rest of the process and consider only the projection of $\mathbb{Q}$ via $(X_t,X_1,X_2)$ on $\mathbb{R}^3$. We call this
measure $\beta$. It is obvious that we can disintegrate $\beta$ into BTP's. Set
$\beta = \int \beta_x \, d\mu(x)$ where $\mu$ is (w.l.o.g.) some probability measure on $\mathbb{R}$
which we can use as an index space. Every $\beta_x$ is a BTP. Now for any $\beta_x$ we consider the measures
$\mu_x = \mathrm{proj}^y\#\beta_x$ and $\nu_x = \mathrm{proj}^z\#\beta_x$,
i.e.\ the last two marginals of $\beta_x$. These are in convex order, so we can find
the maximal martingale transport plan for the cost $(x+y-K)_+$ with marginals $\mu_x,\nu_x$. We
call it $\pi_x$. This is a finite measure on $\mathbb{R}^2$ but not necessarily the right part of a BTP.
But we can (finitely) disintegrate $\pi_x$ again into $\pi_x = \sum_{i=1}^{n_x} a_{x,i}\pi_{x,i}$ with
$a_{x,i} \in \mathbb{R}^+$ and $a_{x,1}+\dots+a_{x,n_x}=1$. The decomposition is chosen such that every
$\pi_{x,i}$ is the right part of a BTP and we have $\int s \, d\mu_x(s) = \int s \, d\pi_{x,i}(s,t)$.
See also Figure \ref{modgraph} for this procedure.

Now we need to reassemble the pieces to arrive at the result we want to show. For every $\pi_{x,i}$ it is
clear how to extend it to a full BTP which we will denote by $\tilde{\pi}_{x,i}$. By definition it
can never be the case that $\tilde{\pi}_{x,i}$ fulfills (i) in Lemma \ref{domlem}.

Now we can calculate
\begin{align*}
&\int (x_1+x_3 - K)_+ \, d\beta(x_1,x_2,x_3)\\
& \qquad= \int\int (x_1+x_3 - K)_+ \, d\beta_x(x_1,x_2,x_3) d\mu(x) \\
& \qquad=\int\int (x_1+x_3 - K)_+ \, d\tilde{\pi}_x(x_1,x_2,x_3) d\mu(x) \\
& \qquad= \int\sum_{i=1}^{n_x} \int (x_1+x_3 - K)_+ \, d\tilde{\pi}_{x,i}(x_1,x_2,x_3) d\mu(x) \\
& \qquad\leq \int\sum_{i=1}^{n_x} \int (x_2+x_3 - K)_+ \, d\pi_{x,i}(x_2,x_3) d\mu(x) \\
& \qquad= \int (x_2+x_3 - K)_+ \, d\pi(x_2,x_3)
\end{align*}
which shows the result. For the inequality we used that $\tilde{\pi}_{x,i}$ fulfills (ii) in Lemma \ref{domlem}.
\begin{figure}
\begin{tabular}{ccccc}
\parbox[c]{3cm}{
\resizebox{3cm}{4cm}{
\begin{tikzpicture}[every node/.style={draw,circle,fill=black,inner sep=0pt,minimum size=3pt}]
\draw (0,-4) -- (0,4);
\draw (2,-4) -- (2,4);
\draw (4,-4) -- (4,4);
\node (x) at (0,0) {};
\node (y1) at (2,2) {};
\node (y2) at (2,-2) {};
\node (z1) at (4,3) {};
\node (z2) at (4,1) {};
\node (z3) at (4,-1) {};
\node (z4) at (4,-3) {};
\draw (x) -- (y1);
\draw (x) -- (y2);
\draw (y1) -- (z1);
\draw (y1) -- (z3);
\draw (y2) -- (z2);
\draw (y2) -- (z4);
\end{tikzpicture}}} &
$\longrightarrow$&
\parbox[c]{3cm}{
\resizebox{3cm}{4cm}{
\begin{tikzpicture}[every node/.style={draw,circle,fill=black,inner sep=0pt,minimum size=3pt}]
\draw (0,-4) -- (0,4);
\draw (2,-4) -- (2,4);
\draw (4,-4) -- (4,4);
\node (x) at (0,0) {};
\node (y1) at (2,2) {};
\node (y2) at (2,-2) {};
\node (z1) at (4,3) {};
\node (z2) at (4,1) {};
\node (z3) at (4,-1) {};
\node (z4) at (4,-3) {};
\draw (x) -- (y1);
\draw (x) -- (y2);
\draw (y1) -- (z1);
\draw (y1) -- (z2);
\draw (y1) -- (z4);
\draw (y2) -- (z1);
\draw (y2) -- (z3);
\draw (y2) -- (z4);
\end{tikzpicture}}} &
$\longrightarrow$ &
\begin{tabular}{c}
\resizebox{1.5cm}{2cm}{
\begin{tikzpicture}[every node/.style={draw,circle,fill=black,inner sep=0pt,minimum size=3pt}]
\draw (0,-4) -- (0,4);
\draw (2,-4) -- (2,4);
\draw (4,-4) -- (4,4);
\node (x) at (0,0) {};
\node (y1) at (2,2) {};
\node (y2) at (2,-2) {};
\node (z1) at (4,3) {};
\node (z2) at (4,1) {};
\node (z3) at (4,-1) {};
\node (z4) at (4,-3) {};
\draw (x) -- (y1);
\draw (x) -- (y2);
\draw (y1) -- (z1);
\draw (y1) -- (z2);
\draw (y2) -- (z3);
\draw (y2) -- (z4);
\end{tikzpicture}} \\ $+$ \\
\resizebox{1.5cm}{2cm}{
\begin{tikzpicture}[every node/.style={draw,circle,fill=black,inner sep=0pt,minimum size=3pt}]
\draw (0,-4) -- (0,4);
\draw (2,-4) -- (2,4);
\draw (4,-4) -- (4,4);
\node (x) at (0,0) {};
\node (y1) at (2,2) {};
\node (y2) at (2,-2) {};
\node (z1) at (4,3) {};
\node (z2) at (4,1) {};
\node (z3) at (4,-1) {};
\node (z4) at (4,-3) {};
\draw (x) -- (y1);
\draw (x) -- (y2);
\draw (y1) -- (z1);
\draw (y1) -- (z4);
\draw (y2) -- (z1);
\draw (y2) -- (z4);
\end{tikzpicture}}
\end{tabular} \\
$\beta_x$ & & $\tilde{\pi}_x$ & & $\tilde{\pi}_{x,1}+\tilde{\pi}_{x,2}$
\end{tabular}
\caption{Modification Steps of $\beta$}
\label{modgraph}
\end{figure}
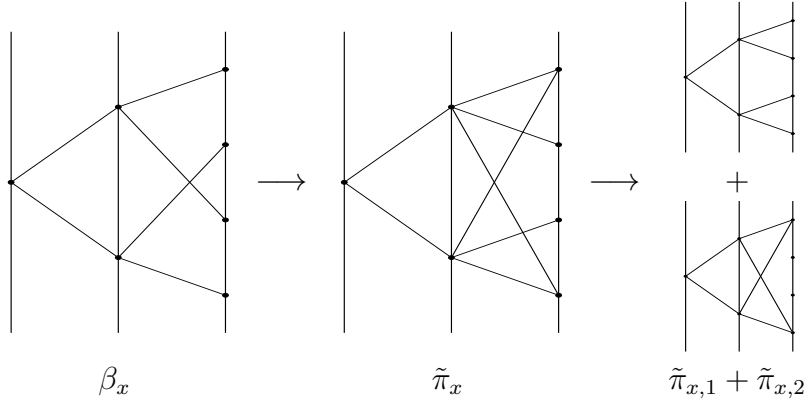
\end{proof}

Before we start the proof of Lemma \ref{domlem} we want to characterise some non-optimal 
parts of maximising transport plans for $|x+y|$ similar to Corollary \ref{maxcor}. We will do this in
another corollary to Lemma \ref{maxlem}. The proof again consists of comparing the values of the
function $f$ (defined in Lemma \ref{maxlem}) at two points.

\begin{corollary}
\label{illcor}
Let $\pi$ be some maximising martingale transport for marginals $\mu \preceq \nu$. Let $\Gamma$ be its
monotonicity set as in Lemma \ref{varlem} and $(x,y^-),(x,y^+),
(x',y') \in \Gamma$ with $y^- < y' < y^+$. The following constellations are not possible ((I1) and (I2) are
just restatements of Corollary \ref{maxcor} without the restriction $0 \leq x,x'$):

\begin{enumerate}[($\text{I}$1)]
\item $x' < x$, $y^- < -x' \leq y'$;
\item $x < x'$, $y' \leq -x' < y^+$;
\item $x' < x$, $-x \leq y^- < -x' < y^+$;
\item $x < x'$, $y^- < -x' < y^+ \leq -x$.
\end{enumerate}
\end{corollary}

The proof of Lemma \ref{domlem} is rather circuitous. We try to improve readability by splitting it
in a computational part and an argumentative part.
We will formulate two auxiliary Lemmas. One which deals with
the computational part of the proof of Lemma \ref{domlem} by enumerating cases in which
the right part dominates the left part. The other one gets rid of a few cases by stating that
the properties formulated in Lemma \ref{domlem} are invariant under symmetries.

\begin{lemma}
\label{symlem}
Let $\pi$ be some BTP. Define the mirrored BTP $\pi'$ by setting
$\pi'(A) := \pi(-A)$ for any measurable $A \subseteq \mathbb{R}^3$ where
we set $-A := \{(-x,-y,-z) : (x,y,z) \in A\}$. Then it holds that
\begin{itemize}
\item $\pi$ fulfills $\int |x+y|\, d\pi_l(x,y) \leq \int |x+y|\,d\pi_r(x,y)$ if and only if
$\pi'$ fulfills $\int |x+y|\, d\pi'_l(x,y) \leq \int |x+y|\,d\pi'_r(x,y)$.
\item $\pi_r$ is an optimal martingale transport if and only if $\pi'_r$ is an optimal martingale
transport.
\end{itemize}
\end{lemma}

\begin{lemma}
\label{comlem}
Let $\pi$ be some BTP. It fulfills $\int |x+y|\, d\pi_l(x,y) \leq \int |x+y|\,d\pi_r(x,y)$
if any of the following holds:
\begin{enumerate}[(L1)]
\item $y^{-+} \leq y^{+-}$;
\item $x^-+y^{+-} \geq 0$, $x^-+y^{--} \geq 0$;
\item $x^+ + y^{++} \leq 0$, $x^+ + y^{-+} \leq 0$;
\item $x^+ + y^{-+} \leq 0$, $x^+ + y^{+-} \leq 0$, $x^+ + y^{++} \geq 0$;
\item $x^- + y^{+-} \geq 0$, $x^- + y^{-+} \geq 0$, $x^- + y^{--} \leq 0$;
\item $x^- + y^{-+} \geq 0$, $x^-+y^{--} \leq 0$, $x^+ + y^{+-} \geq 0$, 
$y^{--} \leq y^{+-} \leq y^{++} \leq y^{-+}$, 
$(1-\lambda_-)(x^+-x^-) \leq x^+ + y^{+-}$,
$(1-\lambda_+)(x^+-x^-) \leq x^+ + y^{++}$ for
$\lambda_{\pm}$ such that $\lambda_{\pm}y^{--}+(1-\lambda_{\pm})y^{-+} = y^{+\pm}$;
\item $x^- + y^{-+} \geq 0 \geq x^- + y^{--}$, $x^- + y^{++} \geq 0$, $x^+ + y^{+-} \leq 0$, $y^{+-} \leq y^{--} \leq y^{++} \leq y^{-+}$,
$\lambda(x^+-x^-) \leq -(x^-+y^{--})$ with $\lambda$ such that
$\lambda y^{+-} + (1-\lambda)y^{++} = y^{--}$;
\item $y^- := y^{--} = y^{+-}$, $x^+ + y^- \leq 0$, $x^- + y^{-+} \geq 0$, $x^+ + y^{++} \geq 0$, $y^{++} \leq y^{-+}$,
$(1-\lambda)(x^+-x^-) \leq x^+ + y^{++}$ for
$\lambda$ such that $\lambda y^{-}+(1-\lambda)y^{-+} = y^{++}$;
\item $y^{-+} = y^{++}$, $y^{--} = y^{+-}$.
\end{enumerate}
\end{lemma}

\begin{proof}
Before we start calculating the desired result for these cases, we want to note
a few identities which will be used again and again in the subsequent computations
\begin{align*}
(x^+ + y^{+\pm}) - (x+y^{+\pm}) &= \lambda^-(x^+-x^-), \\
(x^- + y^{-\pm}) - (x+y^{-\pm}) &= -\lambda^+(x^+-x^-), \\
(x^+ + y^{+\pm}) + (x+y^{+\pm}) &= 2(x^+ + y^{+\pm}) - \lambda^-(x^+-x^-),\\
(x^- + y^{-\pm}) + (x+y^{-\pm}) &= 2(x^- + y^{-\pm}) + \lambda^+(x^+-x^-).
\end{align*}
We will not do the whole string of arithmetic manipulations every time as these are
essentially straightforward. 
\begin{enumerate}[(L1)]
\item Calculate:
\begin{align*}
& \int |x+y|\,d\pi_r(x,y) - \int |x+y|\, d\pi_l(x,y) \\ 
& \quad=\lambda^{-}\lambda^{--}(|x^- + y^{--}| - |x+y^{--}|) \\
& \qquad+\lambda^{-}\lambda^{-+}(|x^- + y^{-+}| - |x+y^{-+}|)\\
& \qquad + \lambda^{+}\lambda^{+-}(|x^+ + y^{+-}| - |x+y^{+-}|) \\
& \qquad+\lambda^{+}\lambda^{++}(|x^+ + y^{++}| - |x+y^{++}|) \\
& \quad \geq \lambda^+\lambda^-(x^+-x^-)(-\sgn(x+y^{--})\lambda^{--}-\sgn(x+y^{-+})\lambda^{-+} \\
& \qquad +\sgn(x+y^{+-})\lambda^{+-}+\sgn(x+y^{++})\lambda^{++}) \\
& \quad \geq \lambda^+\lambda^-(x^+-x^-)(\sgn(x+y^{+-})-\sgn(x+y^{-+})) \geq 0 \text{.}
\end{align*}
\item Note that the given conditions imply that all the terms occuring as values for $|x+y|$ are 
non-negative yielding with the above given identities:
\begin{align*}
& \int |x+y|\,d\pi_r(x,y) - \int |x+y|\, d\pi_l(x,y) \\ 
& \qquad= \lambda^+\lambda^-(x^+ - x^-)(\lambda^{++} + \lambda^{+-} - \lambda^{-+}-\lambda^{--}) = 0.
\end{align*}
\item This follows from (L2) by Lemma \ref{symlem}.
\item From the conditions we can deduce $x^-+y^{--} \leq x+y^{--},x^-+y^{-+} \leq x+y^{-+} \leq x^++y^{-+} \leq 0$
and $x+y^{+-} \leq x^+ + y^{+-} \leq 0$. For $x+y^{++} \geq 0$ we obtain
\begin{align*}
& \int |x+y|\,d\pi_r(x,y) - \int |x+y|\, d\pi_l(x,y) \\ 
& \quad= \lambda^+\lambda^-(x^+ - x^-)(-\lambda^{++} - \lambda^{+-} + \lambda^{-+}+\lambda^{--}) \\
&\qquad+2\lambda^+\lambda^{++}(x^+ + y^{++})\geq 0 \text{.}
\end{align*}
For $x+y^{++} \leq 0$ we get instead
\begin{align*}
& \int |x+y|\,d\pi_r(x,y) - \int |x+y|\, d\pi_l(x,y) \\ 
& \qquad= \lambda^+\lambda^-(x^+ - x^-)(\lambda^{++} - \lambda^{+-} + \lambda^{-+} + \lambda^{--}) \geq 0 \text{.}
\end{align*}
\item This follows from (L4) by Lemma \ref{symlem}.
\item Of course we have here $x+y^{-+} \geq x^-+y^{-+} \geq 0$ and $x^+ + y^{++} \geq x^++y^{+-} \geq 0$.
We distinguish the signs for the other three binary expressions:
\begin{itemize}
\item $x+y^{--} \leq x+y^{+-} \leq x+y^{++} \leq 0$:
\begin{align*}
& \int |x+y|\,d\pi_r(x,y) - \int |x+y|\, d\pi_l(x,y) \\ 
& \quad= \lambda^+\lambda^-(x^+ - x^-)(-\lambda^{++} - \lambda^{+-} - \lambda^{-+}+\lambda^{--}) \\
&\qquad+ 4\lambda^+x^+ \\
& \quad = 2\lambda^+(2x^+ - \lambda^-\lambda^{-+}(x^+ - x^-))\text{.}
\end{align*}
Using the inequalities with $\lambda_-$ and $\lambda_+$ we obtain
$(\lambda^{+-}(1-\lambda_-) + \lambda^{++}(1-\lambda_+))(x^+-x^-) \leq 2x^+$ which we can use in the
above expression to obtain
\begin{align*}
&2x^+ - \lambda^-\lambda^{-+}(x^+ - x^-) \\
&\qquad \geq (\lambda^{+-}(1-\lambda_-) + \lambda^{++}(1-\lambda_+) - \lambda^-\lambda^{-+})(x^+-x^-)\text{.}
\end{align*}
We will now show that $\lambda^{+-}(1-\lambda_-) + \lambda^{++}(1-\lambda_+) - \lambda^-\lambda^{-+} \geq 0$
to obtain our desired result. A sufficient condition for this is that the following expression is non-negative:
\begin{align*}
&(x^+-x^-)((y^{++}-x^+)(y^{+-}-y^{--})\\
&\qquad+(x^+-y^{+-})(y^{++}-y^{--}))\\
&\qquad-(x^+-x)(y^{++}-y^{+-})(x^--y^{--}) \\
&\quad\geq(x^+-x)((y^{++}-x^+)(y^{+-}-y^{--})\\
&\qquad+(x^+-y^{+-})(y^{++}-y^{--})\\
&\qquad-(y^{++}-y^{+-})(x^--y^{--})) \\
&\quad=(x^+-x)(x^+-x^-)(y^{++}-y^{+-}) \geq 0\text{.}
\end{align*}
\item $x+y^{--} \leq x+y^{+-} \leq 0 \leq x+y^{++}$:
\begin{align*}
& \int |x+y|\,d\pi_r(x,y) - \int |x+y|\, d\pi_l(x,y) \\ 
& \quad= \lambda^+\lambda^-(x^+ - x^-)(\lambda^{++} - \lambda^{+-} - \lambda^{-+}+\lambda^{--}) \\
&\qquad+ 2\lambda^+\lambda^{+-}(x^++y^{+-}) \\
& \quad = 2\lambda^+(\lambda^-(x^+ - x^-)(\lambda^{++} - \lambda^{-+})+\lambda^{+-}(x^++y^{+-}))\text{.}
\end{align*}
Using the condition for $\lambda_+$, we get
\begin{align*}
&\lambda^-(x^+ - x^-)(\lambda^{++} - \lambda^{-+})+\lambda^{+-}(x^++y^{+-}) \\
&\qquad \geq (\lambda^-(\lambda^{++}-\lambda^{-+})+\lambda^{+-}(1-\lambda_+))(x^+-x^-)\text{.}
\end{align*}
It remains therefore to show that $\lambda^-(\lambda^{++}-\lambda^{-+})+\lambda^{+-}(1-\lambda_+) \geq 0$.
It is again sufficient that the following expression is non-negative:
\begin{align*}
&(x^+-x)((x^+-y^{+-})(y^{-+}-y^{--})\\
&\qquad-(x^--y^{--})(y^{++}-y^{+-}))\\
&\qquad+(x^+-x^-)(y^{++}-x^+)(y^{+-}-y^{--}) \\
&\quad\geq(x^+-x)((x^+-y^{+-})(y^{-+}-y^{--})\\
&\qquad-(x^--y^{--})(y^{++}-y^{+-})\\
&\qquad+(y^{++}-x^+)(y^{+-}-y^{--})) \\
&\quad=(x^+-x)(x^+(y^{-+}-y^{+-})-x^-(y^{++}-y^{+-})\\
&\qquad+y^{+-}(y^{++}-y^{-+})) \\
&\quad\geq(x^+-x)(x^+-x^-)(y^{++}-y^{+-}) \geq 0\text{.}
\end{align*}
In the last step we used that $y^{-+} \geq y^{++}$ and that we have $y^{+-} \leq 0$ as well as 
$y^{++}-y^{-+} \leq 0$.
\item $x+y^{--} \leq 0 < x+y^{+-} \leq x+y^{++}$:
\begin{align*}
& \int |x+y|\,d\pi_r(x,y) - \int |x+y|\, d\pi_l(x,y) \\ 
& \quad= \lambda^+\lambda^-(x^+ - x^-)(\lambda^{++} + \lambda^{+-} - \lambda^{-+}+\lambda^{--}) \geq 0\text{.}
\end{align*}
\item $0 < x+y^{--} < x+y^{+-} < x+y^{++}$:
\begin{align*}
& \int |x+y|\,d\pi_r(x,y) - \int |x+y|\, d\pi_l(x,y) \\ 
& \quad= \lambda^+\lambda^-(x^+ - x^-)(\lambda^{++} + \lambda^{+-} - \lambda^{-+}-\lambda^{--}) \\
& \qquad -2\lambda^-\lambda^{--}(x^-+y^{--})\geq 0\text{.}
\end{align*}
\end{itemize}
\item \begin{itemize}
\item $x+y^{--} \leq 0$:
\begin{align*}
& \int |x+y|\,d\pi_r(x,y) - \int |x+y|\, d\pi_l(x,y) \\ 
& \quad= \lambda^+\lambda^-(x^+ - x^-)(\lambda^{++} - \lambda^{+-} - \lambda^{-+}+\lambda^{--}) \\
& \quad= 2\lambda^+\lambda^-(x^+ - x^-)(\lambda^{++} - \lambda^{-+})\text{.}
\end{align*}
We need to establish $\lambda^{++} \geq \lambda^{-+}$ or the sufficient condition
$(x^+-y^{+-})(y^{-+}-y^{--}) \geq (x^--y^{--})(y^{++}-y^{+-})$ which can be shown by noting
\begin{align*}
& (x^+-y^{+-})(y^{-+}-y^{--}) - (x^--y^{--})(y^{++}-y^{+-})  \\
& \quad = (x^+-y^{+-})y^{-+} + (y^{++} - x^+)y^{--}+x^-(y^{+-}-y^{++}) \\
& \quad \geq (x^+ - y^{+-})y^{++} + (y^{++}-x^+)y^{+-}+x^-(y^{+-}-y^{++}) \\
& \quad= (x^+-x^-)(y^{++}-y^{+-}) \text{.}
\end{align*}
\item 0 < $x+y^{--}$:
\begin{align*}
& \int |x+y|\,d\pi_r(x,y) - \int |x+y|\, d\pi_l(x,y) \\
& \quad= \lambda^+\lambda^-(x^+ - x^-)(\lambda^{++} - \lambda^{+-} - \lambda^{-+}-\lambda^{--}) \\
& \qquad -2\lambda^-\lambda^{--}(x^-+y^{--}) \\
& \quad= 2\lambda^-(-\lambda^{+-}\lambda^+(x^+-x^-)-\lambda^{--}(x^-+y^{--})) \text{.}
\end{align*}
Using the condition with $\lambda$ we obtain
\begin{align*}
&-\lambda^+\lambda^{+-}(x^+-x^-)-\lambda^{--}(x^-+y^{--}) \\
& \quad\geq (\lambda^{--}\lambda-\lambda+\lambda^{+-})(x^+-x^-) \\
& \quad \geq (\lambda^{--}\lambda-\lambda^{+-})(x^+-x^-) \text{.}
\end{align*}
Thus, it remains to show that $\lambda^{--}\lambda-\lambda^{+-} \geq 0$. By substituting the values of the variables
in this expression, this is equivalent to showing that the following expression is non-negative:
\begin{align*}
& (y^{-+}-x^-)(y^{++}-y^{--})-(y^{++}-x^+)(y^{-+}-y^{--}) \\
& \quad = (x^+-y^{--})y^{-+} + (y^{--}-x^-) -(x^+-x^-)y^{--} \\
& \quad \geq (x^+-x^-)(y^{++}-y^{--}) \geq 0 \text{.}
\end{align*}
Here we used in the last step that in this case $y^{-+} \geq y^{++}$ holds.
\end{itemize}
\item \begin{itemize}
\item $x+y^{++} \geq 0$:
\begin{align*}
& \int |x+y|\,d\pi_r(x,y) - \int |x+y|\, d\pi_l(x,y) \\ 
& \quad= 2\lambda^+\lambda^-(x^+ - x^-)(\lambda^{++} - \lambda^{-+}) \text{.}
\end{align*}
as above this follows from $(x^+-y^-)(y^{-+}-y^-) \geq (x^--y^-)(y^{++}-y^-)$ which is
obvious in this case.
\item $x+y^{++} \leq 0$:
\begin{align*}
& \int |x+y|\,d\pi_r(x,y) - \int |x+y|\, d\pi_l(x,y) \\
& \quad= \lambda^+\lambda^-(x^+ - x^-)(-\lambda^{++} - \lambda^{+-} - \lambda^{-+}-\lambda^{--}) \\
& \qquad +2\lambda^+\lambda^{++}(x^++y^{++}) \\
& \quad\geq 2\lambda^+(x^+-x^-)(\lambda^{++}(1-\lambda)-\lambda^-\lambda^{-+})\text{.}
\end{align*}
Easy substitutions give
\begin{align*}
\lambda^{++}(1-\lambda)-\lambda^-\lambda^{-+} &\geq \lambda^{++}(1-\lambda)-\lambda^{-+} \\
&= \frac{x^+ - x^-}{y^{-+}-y^-} \geq 0 \text{.}
\end{align*}
\end{itemize}
\item This follows from the cases we already proved. By Lemma \ref{symlem} we restrict ourselves to
$0 \leq x^- \leq x^+$ and $x^- \leq 0 \leq x^+$ with $|x^-| \leq x^+$.

In the first of these cases we can either have $-x^- \leq y^-$ which gives us a special case of
(L2). Or we can have $-x^+ \leq y^- \leq -x^-$ which is a special case of (L6). Finally we can have
$y^- \leq -x^+$ which is a special case of (L8).

In the second case we have to distinguish between $-x^+ \leq y^-$ which gives a special case
of (L6) and $y^- \leq -x^+$ which is again a special case of (L8).
\end{enumerate}
\end{proof}

\begin{proof}[Proof of Lemma \ref{domlem}]
We can now do a straightforward analysis of every possible BTP. This leads to a somewhat long list of cases
that we will have to enumerate individually.

First note that for $x^-=x^+$ we have $\pi_l=\pi_r$ and therefore (ii) holds trivially. The same is the case
for $x=x^-$ and $x=x^+$. Therefore we can always assume $x^- < x < x^+$ from now on.

From Lemma \ref{comlem} (L1) we already know that whenever we have $y^{-+} \leq y^{+-}$, we also
have that (ii) holds, so we  have dealt with these cases.

Next we want to look at BTP's such that $y^{--} < y^{+-} < y^{-+} < y^{++}$. From Lemma \ref{symlem}
we see that it is sufficient to deal with those cases where $0 \leq x^- < x^+$ or $x^- \leq 0 \leq x^+$
with $|x^-| \geq x^+$.

In the first case, we see that for $y^{+-} < -x^-$ the BTP fulfills (i) by Corollary \ref{illcor} (I1)
(with elements $(x^+,y^{++}),(x^+,y^{+-}),(x^-,y^{-+})$). Otherwise we have
$x^-+y^{+-} \geq 0$. As we trivially have $x^-+y^{-+} \geq 0$ (from $0 \leq x^- \leq x^{-+}$)
the remaining BTP's will fulfill (ii) by either (L2) or (L5) of Lemma \ref{comlem} depending
on the sign of $x^-+y^{--}$.

In the second case a BTP fulfills (i) if $-x^+ < y^{-+} \leq -x^-$ holds which can be seen from
Corollary \ref{illcor} (I4) with elements $(x^-,y^{--}),(x^-,y^{-+})$, $(x^+,y^{+-})$. This is also
the case for $-x^- \leq y^{-+}$ by Corollary \ref{illcor} (I1) with elements $(x^+,y^{++}),(x^+,y^{+-}),(x^-,y^{-+})$.
Otherwise we have $x^++y^{+-} < x^++y^{-+} \leq 0$ and trivially $x^++y^{++} \geq 0$. Therefore
(ii) holds by \ref{comlem} (L4).

The next type we want to distinguish, consists of the BTP's with $y^{--} \leq y^{+-} \leq y^{++} \leq y^{-+}$. 
If we have $y^{+-} = y^{++} = y^{-+}$ we are back in the first type of plans we discussed. Therefore
we can assume that we only have $y^{--} < y^{+-} = y^{++} < y^{-+}$ or equality for the two lower or the
two larger elements if $y^{+-} < y^{++}$. We note that equality for both of them is covered in Lemma
\ref{comlem} (L9), so we assume that at most one of these pairs is equal.

First we take $x^- < x^+ \leq 0$. For $y^{-+} \leq -x^+$ we have $x^+ + y^{++} \leq x^+ + y^{-+} \leq 0$
so the BTP fulfills (ii) by Lemma \ref{comlem} (L3). If we
instead have $y^{-+} > -x^+$ and $y^{+-} > y^{--}$ the BTP fulfills (i) by Corollary \ref{illcor} (I2)
with elements $(x^-,y^{--}),(x^-,y^{-+}),(x^+,y^{+-})$.
For $y^{+-} = y^{--}$ we also have (i) for
$y^{++} \leq -x^+$ by Corollary \ref{illcor} (I2) with elements $(x^-,y^{--}),(x^-,y^{-+}),(x^+,y^{++})$.
For $y^{++} < y^{-+} \leq -x^-$ we have (i) by Corollary \ref{illcor} (I4) with those same elements.
Lemma \ref{comlem} (L8) gives (i) or (ii) for
$y^{-+} \geq -x^-$. We clearly have to have $x^+ + y^{+-} \leq 0$, $x^+ + y^{++} \geq 0$ and
$x^-+y^{-+} \geq 0$. If (i) does not hold, we have from Lemma \ref{maxlem} that
\begin{align*}
0 &\leq \lambda(|x^-+y^{--}|-|x^++y^{--}|) \\
& \qquad + (1-\lambda)(|x^-+y^{-+}|-|x^++y^{-+}|) \\
&\qquad +(|x^++y^{++}|-|x^-+y^{++}) \\
&\quad = 2(x^++y^{++})-2(1-\lambda)(x^+ - x^-) 
\end{align*}
holds. Therefore Lemma \ref{comlem} (L8) gives us (ii).

Now we take $0 \leq x^- < x^+$. For $-x^- \leq y^{--}$ we have $x^-+y^{+-} \geq x^-+y^{--} \geq 0$
and (ii) holds by Lemma \ref{comlem} (L2).
For $y^{--} \leq -x^-$ and $y^{+-} \geq -x^+$ we have (ii) by Lemma \ref{comlem} (L6)
as long as the plan does not fulfill (i). This is because we clearly
have $x^-+y^{-+} \geq 0$, $x^-+y^{--} \leq 0$ and $x^++y^{+-} \geq 0$. Furthermore any plan which
does not fulfill (i) has to fulfill by Lemma \ref{maxlem} that
\begin{align*}
0 &\leq \lambda_-(|x^-+y^{--}|-|x^++y^{--}|) \\
& \qquad + (1-\lambda_-)(|x^-+y^{-+}|-|x^++y^{-+}|) \\
&\qquad +(|x^++y^{+-}|-|x^-+y^{+-}) \\
&\quad = 2(x^++y^{+-})-2(1-\lambda_-)(x^+ - x^-)
\end{align*}
and the second condition from (L6) which has been calculated in the last paragraph (here $\lambda_+=\lambda$).
If we instead have $y^{--} < y^{+-} \leq -x^+$ we have (i) (by
Corollary \ref{illcor} (I2) with $(x^-,y^{--})$,$(x^-,y^{-+})$,$(x^+,y^{+-})$.
For $y^{--} = y^{+-} \leq -x^+$ we can deduce ((i) or) (ii) from Lemma \ref{comlem} (L8).

The last subcase is $x^- \leq 0 \leq x^+$. 
As above we have (i) or (ii) for $y^{-+} \geq -x^-$
and $y^{+-} \geq -x^+$ by Lemma \ref{comlem} (L6) as $y^-+y^{--} \leq 0$ is fulfilled
trivially.
For $y^{--} < -x^+$ and $y^{-+} \leq -x^-$ the BTP
fulfills (i) as soon as not both $y^{--} = y^{+-}$ and $y^{-+} = y^{++}$ hold by Corollary $\ref{illcor}$ (I4)
with elements $(x^-,y^{--}),(x^-,y^{-+})$, $(x^+,y^{++})$ (or $(x^+,y^{+-})$).
 For $y^{--} < y^{+-} \leq -x^+$ the BTP fulfills (i) by Corollary \ref{illcor} (I2) with elements
$(x^-,y^{--}),(x^-,y^{-+}),(x^+,y^{+-})$.
If we instead have $y^{--} = y^{+-} \leq -x^+$ then (i) or (ii) follows from Lemma \ref{comlem} (L8)
as $x^+y^{+-} \leq 0$ holds by assumption, $x^++y^{++} \geq 0$ holds trivially in this case and
$y^{-+} \geq -x^-$ has to hold because we already showed (i) for the other possibility.

The cases with $y^{+-} \leq y^{--} \leq y^{-+} \leq y^{++}$ are again mirrored from this and can be safely omitted.

The last type we need to discuss are the BTP's with $y^{+-} < y^{--} < y^{++} < y^{-+}$. Again
we get from Lemma \ref{symlem} that it is sufficient to deal with those cases 
where $0 \leq x^- < x^+$ or $x^- \leq 0 \leq x^+$ with $|x^-| \leq x^+$.

In the first case we get (ii) from Lemma \ref{comlem} (L2) for $y^{+-} \geq -x^-$.
For $y^{+-} < -x^- \leq y^{--}$ and $-x^+ \leq y^{+-} < (y^{--} \leq) -x^-$ 
the BTP fulfills (i) by Corollary \ref{illcor} (I1) or (I3) respectively with 
$(x^+,y^{++}),(x^+,y^{+-})$, $(x^-,y^{--})$. Furthermore, it remains to check that (i) or (ii) holds for
$x^+ + y^{+-}\leq 0$, $x^-+y^{-+} > x^-+y^{++} \geq 0$. and $x^-+y^{--} \leq 0$.
If (i) does not hold, we have by Lemma \ref{maxlem} that 
\begin{align*}
0 &\leq \lambda(|x^++y^{+-}|-|x^-+y^{+-}|) \\
&\qquad + (1-\lambda)(|x^++y^{++}|-|x^-+y^{++}|) \\
&\qquad +(|x^-+y^{--}|-|x^++y^{--}|) \\
&=-2\lambda (x^+-x^-)-2(x^-+y^{--})
\end{align*}
holds. Therefore the conditions of Lemma \ref{comlem} (L7) are fulfilled, and (ii) follows.

In the second case we can deal with BTP's with $-x^+ \leq y^{+-}$ by using 
Corollary \ref{illcor} (I3) with $(x^+,y^{+-}),(x^+,y^{++}),(x^-,y^{--})$. Otherwise
the conditions of Lemma \ref{comlem} (L7) are fulfilled again, and we are done.
\end{proof}

\bibliography{paper}
\bibliographystyle{abbrv}

\end{document}